\documentclass[11pt,reqno]{amsart}

\newcommand{\mysection}[1]{\section{#1}
      \setcounter{equation}{0}}
\usepackage{color}

\newtheorem{theorem}{Theorem}[section]
\newtheorem{lemma}[theorem]{Lemma}
\newtheorem{proposition}[theorem]{Proposition}
\newtheorem{corollary}[theorem]{Corollary}

\theoremstyle{definition}
\newtheorem{assumption}{Assumption}[section]
\newtheorem{definition}{Definition}[section]
\newtheorem{example}{Example}[section]

\theoremstyle{remark}
\newtheorem{remark}{Remark}[section]

\newcommand\bC{\mathbb{C}}

\newcommand\bR{\mathbb{R}}

\newcommand\bG{\mathbb{G}}

\newcommand\bH{\mathbb{H}}

\newcommand\fra{\mathfrak{a}}
\newcommand\frb{\mathfrak{b}}

\newcommand{\ga}{\mathfrak{a}}
\newcommand{\gb}{\mathfrak{b}}

\newcommand{\gp}{\mathfrak{p}}
\newcommand{\gq}{\mathfrak{q}}

\newcommand\cB{\mathcal{B}}

\newcommand\cF{\mathcal{F}}

\newcommand\cH{\mathcal{H}}

\newcommand\cK{\mathcal{K}}
\newcommand\cL{\mathcal{L}}
\newcommand\cM{\mathcal{M}}

\newcommand\cP{\mathcal{P}}

\newcommand\cR{\mathcal{R}}

\newcommand\cO{\mathcal{O}}

\makeatletter
 \newcommand{\sumstar}%
 {\operatornamewithlimits{\sum@\kern-.2em\raise1ex\hbox{*}}}
 \makeatother

\renewcommand{\>}{\rangle}

\begin{document}

\title[stochastic finite difference schemes]
{On stochastic finite difference schemes}

\author[I. Gy\"ongy]{Istv\'an Gy\"ongy}
\address{School of Mathematics and Maxwell Institute,
University of Edinburgh,
King's  Buildings,
Edinburgh, EH9 3JZ, United Kingdom}
\email{gyongy@maths.ed.ac.uk}


\begin{abstract}
Finite difference schemes in the spatial variable for degenerate stochastic parabolic PDEs are investigated. 
Sharp results on the rate of $L_p$ and almost sure convergence of the finite difference approximations 
are presented and results on Richardson extrapolation are established 
for stochastic parabolic schemes under smoothness assumptions.  
\end{abstract}

\maketitle

\mysection{Introduction}                        \label{section 1.2.5.12}

We consider finite difference schemes to 
stochastic partial differential equations (SPDEs).  The 
stochastic PDEs we are interested in are linear second order stochastic parabolic equations in the whole $\bR^d$ in the spatial variable.  
They may degenerate and become first order stochastic or deterministic PDEs. 
The finite difference schemes which we investigate are spatial discretizations of such SPDEs. One can view them as (possibly degenerate) infinite systems of stochastic differential equations, whose components describe the time evolution of approximate values at the grid points of the solutions to SPDEs.  Adapting the approach of \cite{K1} we view stochastic finite difference schemes, 
like in \cite{GK4} and 
\cite{G11}, as stochastic 
equations for random fields on the whole $\bR^d$ not only on grids.  

Our aim is to investigate the rate of convergence in the supremum norm of the 
finite difference approximations. We show that under the {\it stochastic parabolicity}  condition, if the coefficients and the data are sufficiently smooth, then 
the solutions to the finite difference schemes admit power series expansions  in terms of $h$, the mesh-size of the grid. The coefficients in these power series
are random fields, independent of $h$, and for any $p>0$ the $p$-th moments of the sup norm of the remainder term is estimated by a power of $h$. 
This is Theorem \ref{theorem main}. Hence for $h\to0$ we get the convergence (and the sharp rate) of the solutions of the finite difference schemes to a random field which is the solution to the corresponding SPDE. Moreover, by Richardson extrapolation we get that the rate of convergence can be accelerated to any high order if one takes appropriate mixtures of approximations corresponding to different grid sizes. In Theorem \ref{theorem 2.15.10.12} we obtain convergence estimates for any (high) $p$-th moments of the 
sup norm of the approximation error. Hence in Theorem \ref{theorem almost sure} 
we get almost sure rate of convergence of the finite difference approximations and of the accelerated approximations in sup norms. 

Theorems \ref{theorem main} and  \ref{theorem 2.15.10.12} are  generalisations of the main results, Theorems 2.3  and 2.5 of  \cite{G11}, which besides the conditions on the smoothness of the coefficients and of the data require also a {\em  smooth factorisation}  condition to be satisfied. 
Namely, in \cite{G11} it is assumed that the matrix 
$\tilde a=\tilde a(\omega,t,x)$ from the 
stochastic parabolicity condition can be written as the product of a sufficiently smooth matrix and its transpose. This assumption, however, may not hold even 
if $\tilde a$ is an infinitely differentiable nonnegative matrix, and therefore 
it strongly restricts the applicability of the results of \cite{G11}. 

The main challenge in the present paper is to estimate the spatial derivatives 
of the solutions to stochastic finite difference schemes without assuming 
the smooth factorisation condition. 
This is achieved by Theorem \ref{theorem estimate}. 

The method of finite differences is one of the basic methods of solving numerically partial differential equations. The rate of convergence of various finite difference schemes 
for elliptic and parabolic PDEs have been studied extensively in the literature when 
the equations are non degenerate, but there are only a few publications dedicated to the numerical analysis of finite difference schemes for degenerate equations.  
Sharp rate of convergence in {\em sup} norm are obtained in \cite{DK}  
for fully discretized degenerate elliptic and parabolic PDEs. The finite difference 
schemes investigated in \cite{DK}  are {\it monotone schemes}. In \cite{GK3} 
for a large class of monotone finite difference schemes (in the spatial variables)  
power series expansions are obtained and Richardson extrapolation is used 
to get accelerated schemes for degenerate elliptic and parabolic PDEs.  
We note that the finite difference schemes we study 
in the present paper are not necessarily monotone, and our main theorems 
extend some of the results of \cite{GK3} to non monotone finite difference 
schemes for degenerate parabolic PDEs.

About a century ago L.F. Richardson had the idea that the order of accuracy of  
an approximation method, which depends on a parameter can be dramatically improved if the approximation calculated by the method admits a power series expansion in the parameter. One need only take appropriate linear combinations of approximations corresponding to different proportions of the 
parameter values to eliminate the lower order terms in the power series to get 
approximations with accuracy of higher order.  
Richardson used this idea to solve numerically some PDEs 
by finite difference methods (see \cite{Ri} and \cite{RG}). He called his method 
{\em a deferred approach to the limit}. It is often called 
{\em Richardson extrapolation} in numerical analysis.  Richardson extrapolation 
is applied by W. Romberg to the trapezium rule to obtain high order  
approximations of definite integrals (see \cite{Rom}). Since then Richardson 
extrapolation has been applied to a wide range of numerical approximations. 
(See, for example, the textbook \cite{MS}, the monograph \cite{Si}  and the survey papers \cite{Br} and \cite{Jo}). 
To show that it is applicable to an approximation method to solve numerically 
a problem under certain conditions 
one has to show the existence of a suitable power series expansion, which can be quite difficult in some situations. A series expansion for  the weak convergence of Euler approximations of SDEs was obtained in \cite{TT} and then in a more general setting in \cite{MT}. The applicability of Richardson extrapolation to accelerate the convergence of 
finite difference schemes in the spatial variable for stochastic PDEs was shown in \cite{GK}, under 
the strong stochastic parabolicity condition, and then 
it was shown in \cite{G11} for degenerate SPDEs under the smooth  
factorisation condition mentioned above. The results from \cite{GK} and \cite{G11} 
have been generalised to finite difference schemes in temporal and spatial variables in  \cite{H1} and \cite{H2}. 

The paper is organised as follows. In the next section we formulate our main 
result, Theorem \ref{theorem main} on power series expansions for stochastic finite 
difference schemes. Hence our results on the $L_p$ and almost sure rate of convergence of the accelerated schemes,  Theorems \ref{theorem 2.15.10.12} and 
\ref{theorem almost sure}, follow easily. In Section \ref{section 1.12.10.12} we formulate also an existence and uniqueness theorem, 
Theorem \ref{theorem 1.14.10.12},  for stochastic parabolic (possibly degenerate) equations. Though this result is known from 
\cite{KR} and \cite{GK}, we give a new proof of it in Section \ref{section tools} 
by construction the solutions to SPDEs via finite difference approximations. 
In Section 3 we present the technical tools to prove 
our key estimate, Theorem \ref{theorem estimate} on stochastic finite difference 
schemes.  
In Section \ref{section coefficients} we prove an existence and uniqueness theorem, 
Theorem  \ref{theorem 5.28.1} for a system of SPDEs by the help of Theorem \ref{theorem 1.14.10.12}. 
Theorem  \ref{theorem 5.28.1} plays a crucial role in identifying the coefficients 
of the expansion in Theorem \ref{theorem main}. We prove Theorem \ref{theorem main} in 
Section \ref{section proof}, adapting a method from \cite{GK4}.

We conclude by introducing some notation. We fix a complete probability space 
$(\Omega,\cF,P)$ and an increasing family 
$\mathbb F=(\cF_t)_{t\geq0}$ of $\sigma$-algebras  $\cF_t\subset\cF$  
throughout the paper. We assume that $\mathbb F$ is right continuous and that 
$\cF_0$ contains all $P$-zero sets. The $\sigma$-algebra of the predictable subsets of $\Omega\times[0,\infty)$ is denoted by $\cP$, and  the 
$\sigma$-algebra of the Borel subsets of $\bR^d$ is denoted by $\cB(\bR^d)$. The notation $|v|$ means the Euclidean norm of $v$ 
for vectors 
$v\in\bR^d$, and it stands for the 
Hilbert-Schmidt norm of $v$ 
for matrices $v\in\bR^{d\times m}$. 
 For the standard basis $e_1$,...,$e_d$ in $\bR^d$ we use the notation 
$$
D_{\alpha}=D_i=\frac{\partial}{\partial x^i}=\partial_{e_i}
\quad 
\text{for $\alpha=i\in\{1,...,d\}$}, 
$$
and we use $D_{\alpha}=\partial_{e_{\alpha}}$ 
for the identity operator when $\alpha=0$. 
For vectors $\lambda=(\lambda^{1},\dots,\lambda^d)\in\bR^d$ we use the notation 
$\partial_{\lambda}\varphi=\sum_{i=1}^{d}\lambda^iD_i\varphi$ for directional derivatives of functions $\varphi$. For an integer $k>0$ the notation 
$|D^k\varphi|$ is used for the Euclidean norm of the vector whose components 
are (in some ordering) 
the partial derivatives of $\varphi$ of order $k$, and $|D^k\varphi|$ 
means $|\varphi|$ for $k=0$. 
For an integer $m\geq0$ we use the notation $H^m=W^m_2$ for 
the Sobolev space defined as the completion of $C_0^{\infty}=C_0^{\infty}(\bR^d)$,  the smooth functions $\varphi$ with compact support  
on $\bR^d$,  in the norm $|\varphi|_m$ defined by 
$$
|\varphi|_{m}^2=|\varphi|^2_{H^m}=\sum_{j=0}^k\int_{\bR^d}|D^j\varphi(x)|^2\,dx. 
$$
The Sobolev spaces $H^m(l_2)=W^m_2(l_2)$ of functions with values in 
$l_2=\{(c_n)_{n=1}^{\infty}\in\bR^{\infty}:\sum_j|c_j|^2<\infty\}$ are defined analogously, and the notation $|\varphi|_m=|\varphi|_{H^m(l_2)}$ is used 
for the norm of $\varphi$ in $H^m(l_2)$. The inner product of functions $\psi$ 
and $\varphi$ in $H^0=L_2(\bR^d)$ is denoted by $(\psi,\varphi)$. 
The summation convention with respect to repeated indices with values in discrete sets is used thorough the paper, unless it is otherwise indicated at some expressions.

\mysection{Formulation of the main results}          \label{section 1.12.10.12}

We consider the stochastic PDE
\begin{align}                            
du_t(x)=&(a^{\alpha\beta}_t(x)
D_{\alpha}D_{\beta}u_t(x)+f_t(x))\,dt
\nonumber\\
&+(b^{\alpha,r}_t(x)D_{\alpha}u_t(x)
+g^{r}_t(x))\,dw^r_t, 
                                              \label{equation}
\end{align}
for $(t,x)\in[0,T]\times\bR^d=:H_T$ for a fixed $T\in(0,\infty)$, with initial condition 
\begin{equation}                            \label{ini}
u_0(x)=\psi(x)\quad x\in\bR^d, 
\end{equation}
where $(w^r)_{r=1}^{\infty}$ is 
a sequence of independent Wiener martingales  
with respect to $\mathbb F$.  
The coefficients, $a^{\alpha\beta}=a^{\beta\alpha}$ 
and $b^{\alpha}=(b^{\alpha, r})_{r=1}^{\infty}$,  
are $\cP\otimes\cB(\bR^d)$-measurable 
bounded functions on $\Omega\times H_T$, with values in $\bR$ and in 
$l_2$ 
respectively, for every $\alpha,\beta\in\{0,1,...,d\}$.  
The free terms, 
$f=f_t(\cdot)$ and $g=(g^{r}_t(\cdot))_{r=1}^{\infty}$ 
are $H^1$-valued and $H^2(l_2)$-valued adapted processes 
for $t\geq0$, respectively, 
such that almost surely 
$$
\int_0^T|f_t|_{H^1}^2\,dt<\infty,
\quad 
\int_0^T|g_t|_{H^2(l_2)}^2\,dt<\infty. 
$$

For a $\cF_0$-measurable $H^1$-valued initial value $\psi$ the solution 
to \eqref{equation}-\eqref{ini} is defined as follows. 

\begin{definition}                                    \label{definition solution}
An $H^1$-valued adapted weakly 
continuous process $u=(u_t)_{t\in[0,T]}$ 
is a (generalized) solution to 
\eqref{equation}-\eqref{ini} on a stochastic interval $[0,\tau]$ 
for a stopping time $\tau\leq T$, if almost 
surely 
\begin{align}                                         
(u_t,\varphi)=&(\psi,\varphi)+
\int_0^t
\{(\partial_{e_{\alpha}}u_s,
\partial_{-e_{\beta}}(a^{\alpha\beta}_s\varphi))+
(f_s,\varphi)\}\,ds       \nonumber\\
&+\int_0^t(b^{\alpha,r}_s\partial_{e_{\alpha}}u_s+g^{r}_s,\varphi)\,dw_s^r, 
\nonumber
\end{align}
for $t\in[0,\tau]$ and $\varphi\in C_0^{\infty}$, 
where  the summation convention is in force with 
respect to the repeated 
indices $\alpha,\beta\in\{0,1,...,d\}$ and 
$r\in\{0,1,...\}$. 
\end{definition}

We approximate equation \eqref{equation} 
with finite difference schemes in the spatial variable. 
To describe these schemes let $\Lambda_1\subset\bR^d$ be a finite set, 
containing the zero vector, and  
set $\Lambda_0=\Lambda_1\setminus\{0\}$. 
For $h\in\bR\setminus\{0\}$ define the grid 
$$
\bG_{h}=\{h(\lambda_1+...+\lambda_n)
:\lambda_i\in\Lambda_1\cup\{-\Lambda_1\},n=1,2,...\}, 
$$
the finite difference operators $\delta_{h,\lambda}$, 
$\delta_{\lambda}^h$ by 
$$
\delta_{h,\lambda}\varphi(x)
=\frac{1}{h}(\varphi(x+h\lambda)-\varphi(x)), 
\quad \delta^h_{\lambda}
=\frac{1}{2}(\delta_{h,\lambda}+\delta_{-h,\lambda})
=\frac{1}{2}(\delta_{h,\lambda}-\delta_{h,-\lambda}), 
$$
for $\lambda\in\Lambda_0\cup\{-\Lambda_0\}$, 
and let $\delta_{h,\lambda}$ and  
$\delta_{\lambda}^h$ be the identity operator for 
$\lambda=0$. 

For $h\neq0$ we consider the equation 
\begin{equation}                                                                       \label{scheme}
du_t^h(x)=(L^h_tu_t^h(x)+f_t(x))\,dt
+(M^{h,r}_tu_t(x)+g^{h}_t(x))\,dw^{r}_t
\end{equation}
for $t\in[0,T]$ and $x\in\bG_h$, with initial condition  
\begin{equation}                         \label{schemeini}
u_0^h(x)=\psi(x), \quad\text{for $x\in\bG_h$}, 
\end{equation}
where 
\begin{equation}                        \label{1.26.10.12}
L^h_t=\sum_{\lambda,\mu\in\Lambda_1}
\fra_t^{\lambda\mu}\delta^h_{\lambda}\delta^h_{\mu}
+\sum_{\gamma\in\Lambda_0}
(\frak p_t^{\gamma}\delta_{h,\gamma}
-\frak q_t^{\gamma}\delta_{-h,\gamma}), 
\end{equation}
\begin{equation}                        \label{2.26.10.12}
M^{h,r}_t=\sum_{\lambda\in\Lambda_1}
\frb_t^{\lambda,r}\delta_{\lambda}^h, 
\quad r=1,2,..., 
\end{equation}
with real valued 
$\cP\otimes{\mathcal B}(\bR^d)$-measurable 
bounded functions, $\fra^{\lambda\mu}=\fra^{\mu\lambda}$, 
$\frak p^{\gamma}$, $\frak q^{\gamma}$,  
and $l_2$-valued $\cP\otimes{\mathcal B}(\bR^d)$-measurable 
bounded functions 
$\frb^{\lambda}=(b^{\lambda,r})_{r=1}^{\infty}$ 
on $\Omega\times H_T$, 
for all $\lambda$, $\mu\in\Lambda_1$ and $\gamma\in\Lambda_0$.  

Equation \eqref{scheme} is an infinite system of stochastic differential 
equations. We look for its solutions in the space of adapted 
stochastic processes with values in 
$l_{h,2}$, the space of real functions $\phi$ on $\bG_h$ with the norm 
$|\phi|_{l_{h,2}}$ defined by 
$$
|\phi|_{l_{h,2}}^2
:=\sum_{x\in\bG_h}|\varphi(x)|^2h^d<\infty. 
$$
\begin{remark}                                 \label{remark 4.13.10.14}
Due to the boundedness of $\fra^{\lambda\mu}$ 
and $\frb^{\lambda}$ for $\lambda,\mu\in\Lambda_1$, 
it is easy to see 
that for each $h\neq0$ there is a constant $C$ 
such that for all $\phi\in l_{h,2}$
$$
|L^h\phi|_{l_{h,2}}^2\leq C|\phi|_{l_{h,2}}^2,
\quad 
\sum_{r=1}^{\infty}|M^{h,r}\phi|_{l_{h,2}}^2\leq C|\phi|^2_{l_{h,2}}. 
$$
Thus by standard results on SDEs in Hilbert spaces the 
initial value problem 
\eqref{scheme}-\eqref{schemeini} admits a unique solution 
$(u^h_t)_{t\in[0,T]}$, provided almost surely 
\begin{equation}                        \label{3.13.10.12}
|\psi|_{l_{h,2}}^2+ 
\int_0^T|f_t|^2_{l_{h,2}}+|g_t|^2_{l_{h,2}}\,dt<\infty \,\,(a.s.), 
\end{equation}
where $|g_t|^2_{l_{h,2}}=\sum_{r=1}^{\infty}|g_t^r|^2_{l_{h,2}}$. 
\end{remark}

Clearly, for $h\to0$ 
$$
\delta_{h,\lambda}\varphi(x)\to \partial_{\lambda}\varphi(x), 
\quad 
\delta^{h}_{\lambda}\varphi(x)\to \partial_{\lambda}\varphi(x)
$$
for smooth functions $\varphi$ on $\bR^d$. 
Thus in order $L^h$ and $M^{h,r}$ approximate the differential operators 
$$
\cL=a^{\alpha\beta}D_{\alpha}D_{\beta}, \quad\text{and}\quad 
\cM^r=b^{\alpha,r}D_{\alpha}, 
$$
respectively, 
we make the following assumption. 
\begin{assumption}                   \label{assumption 1.12.10.12}
For every $i,j=1,...,d$ we have 
$$
a^{ij}=\sum_{\lambda,\mu\in\Lambda_0}\fra^{\lambda\mu}\lambda^{i}\mu^{j}, 
\quad 
a^{0i}+a^{0i}=\sum_{\lambda\in\Lambda_0}
(\fra^{0\lambda}+\fra^{\lambda0}+\frak p^{\lambda}-\frak q^{\lambda})\lambda^i, 
$$
$$
b^{i}=\sum_{\lambda\in\Lambda_0}\frb^{\lambda}\lambda^i,
\quad 
a^{00}=\frak a^{00}, \quad b^0=\frb^0.
$$
\end{assumption}

\begin{example}
                                            \label{example 5.22.1}
Set $\Lambda =\{e_{0}, e_{1},...,  e_{d}\}$, where $e_{0}=0$
and $e_{i}$ is the $i$th basis vector, 
and  let 
\begin{equation*}
\ga^{e_{\alpha} e_{\beta}}_{t}= a^{\alpha\beta}_{t},\quad
\gb^{e_{\alpha}}_{t}=b^{\alpha}_{t},
\quad \alpha,\beta=0,1,...,d, 
\end{equation*}
\begin{equation*}
\quad\gq^{e_{\gamma}}=\gp^{e_{\gamma}}=0
\quad \gamma=1,...,d.
\end{equation*}
\end{example}

\begin{example}
                                     \label{example 5.22.2}
Take $\Lambda =\{e_{0}, e_{1},...,  e_{d}\}$ with $e_{0}=0$ 
as before,  
and define 
\begin{equation*}
\ga^{e_{\alpha} e_{\beta}}= a^{\alpha\beta},
\quad
\ga^{e_{\alpha} 0}=\ga^{0e_{\alpha}}=0, 
\quad \alpha,\beta=1,...,d, 
\end{equation*}
$$
\gp^{e_{\gamma}}=\frac{1}{2}a^{0\gamma}+\theta^{\gamma}, 
\quad 
\gq^{e_{\gamma}}=-\frac{1}{2}a^{\gamma0}+\theta^{\gamma}, 
\quad \gamma=1,...,d, 
$$
\begin{equation*}
\ga^{00}=a^{00},
\quad  
\gb^{e_{\alpha}}_{t}=b^{\alpha}_{t},
\quad \alpha=0,1,...,d, 
\end{equation*}
where $\theta^1,...,\theta^d$ 
are any constants such that 
$|a^{0,\gamma}|\leq 2\theta^{\gamma}$, 
$|a^{\gamma0}|\leq 2\theta^{\gamma}$ 
for each $\gamma\in\{1,...,d\}$.
\end{example}

To formulate our main results we make further assumptions. 
Let $\frak m\geq0$ be an integer, and let $K\geq0$ be a constant. 
Set $|\Lambda_0|^2=\sum_{\lambda\in\Lambda_0}|\lambda|^2$ and  
\begin{equation}                                                \label{1.15.09.13}
\cK_l^2(t)=\int_0^t(|f_t|^2_l+|g|^2_{l+1})\,dt
\end{equation}
for $t\geq0$ and integers $l\geq0$. 

\begin{assumption}                \label{assumption 1.15.9.12}     
The functions $\fra^{\lambda\mu}$, 
$\frb^{\lambda}$,
$D\frb^{\lambda}$ 
and their partial derivatives in $x\in\bR^d$ 
up to order $\frak m$   
are 
continuous in $x$ 
and they are bounded in magnitude by 
$K$ for all $\lambda$ and $\mu\in\Lambda_1$. 
Moreover, for 
$\lambda,\mu\in\Lambda_0$ 
the partial derivatives in $x$ 
of $\fra^{\lambda\mu}$ up to order 
$\max(\frak m,2)$, of $\frb^{\lambda}$ up 
to order $\max(\frak m+1,2)$, 
and of $\fra^{\lambda0}$, 
$\frak p^{\lambda}$, $\frak q^{\lambda}$ 
up to order 
$\max(\frak m,1)$
are continuous in $x$ 
and in magnitude are bounded by $K$.
\end{assumption}

\begin{assumption}                         \label{assumption 2.15.9.12}                                                 
The initial value $\psi$ is an $H^{\frak m}$-valued 
$\cF_0$-measurable 
random variable, 
$(f_t)_{t\geq0}$ is an $H^{\frak m}$-valued 
predictable process and 
$(g_t)_{t\geq0}$ is 
an $H^{\frak m+1}(l_2)$-valued 
predictable process, 
such that almost surely $\cK_{\frak m}^2(T)<\infty$.  
\end{assumption}

\begin{assumption}                           \label{assumption 3.7.5.12}
For $P\otimes dt\otimes dx$-almost all 
$(\omega,t,x)\in\Omega\times H_T$ 
we have
\begin{itemize} 
\item[(i)]
$
\sum_{\lambda,\mu\in\Lambda_0}
(\fra^{\lambda\mu}-\tfrac{1}{2}\frb^{\lambda,r}\frb^{\mu,r})z_{\lambda}z_{\mu}\geq0
$
for all $z_{\lambda}\in\bR$, $\lambda\in\Lambda_0$;
\item[(ii)] $\frak p^{\lambda}\geq0$, 
$\frak q^{\lambda}\geq0$ for all $\lambda\in\Lambda_0$.  
\end{itemize} 
\end{assumption}
\bigskip

\begin{remark}                           \label{remark 1.13.10.12}
By virtue of 
Sobolev's theorem on embedding $H^{\frak m}$ 
into $C_b$, the space of 
bounded and continuous functions on $\bR^d$, 
if Assumption \ref{assumption 2.15.9.12} holds with 
$\frak m>d/2$, then we can, and will always assume that 
almost surely $\psi$, 
$f_t$ and $g_t$ are continuous functions in $x$ for all $t\in [0,T]$.  
Moreover, in this case 
\eqref{3.13.10.12} also holds, due to a result on  
embedding $H^{\frak m}$ into $l_{h,2}$, 
see Lemma \ref{lemma beagyazas} below.  
Consequently, if Assumption \ref{assumption 2.15.9.12} holds 
with $\frak m>d/2$, then \eqref{scheme}-\eqref{schemeini} admits a unique 
solution $(u^h)_{t\in[0,T]}$. 
\end{remark}

An existence and uniqueness theorem for \eqref{equation}-\eqref{ini}  
reads as follows. 

\begin{theorem}                               \label{theorem 1.14.10.12}
Let Assumptions \ref{assumption 1.12.10.12}, 
\ref{assumption 1.15.9.12},  
\ref{assumption 2.15.9.12} and \ref{assumption 3.7.5.12}
hold with $\frak m\geq1$. Then \eqref{equation}-\eqref{ini} 
has a unique solution $(u_t)_{t\in[0,T]}$. 
Moreover, $u$ is an $H^{\frak m}$-valued 
weakly continuous process, it is strongly continuous 
as an $H^{\frak m-1}$-valued process, and for every $p>0$ 
and stopping time $\tau\leq T$ we have 
\begin{equation}                                     \label{1.15.10.12}
E\sup_{t\leq \tau}|u_t|^p_{l}\leq N(E|\psi|_l^p+E\cK^p_{l}(\tau))  
\end{equation}
for every integer $l\in[0,\frak m]$, 
where $N$ is a constant depending only 
on $K$, $T$, $\frak m$, $d$, $p$ and $|\Lambda_0|$. 
\end{theorem} 

This theorem is a special case of 
Theorem 3.1 from \cite{GK}, 
which improves Theorem 3.1 from \cite{KR}. 
As a by-product of our results on finite difference 
approximations, 
we give a new proof of it in 
Section \ref{section tools}. 

Our aim is to establish   an expansion 
of $u^h$ in the form 
\begin{equation}                                             \label{expansion}
u^h_t(x)
=\sum_{j=0}^k\frac{h^j}{j!}u^{(j)}_t(x)
+h^{k+1}r_t^h(x)  
\end{equation}                    
for integers $k\geq0$, 
where $u^{(0)}$ is the solution of 
\eqref{equation}-\eqref{ini}, $u^{(1)}$,...,$u^{(k)}$  
are random fields on $H_T$, independent of $h$, and 
$r^h$ is a random field on $H_T$ such that 
\begin{equation}                              \label{r}
E\sup_{t\leq T}\sup_{x\in\bG_h}|r^h_t(x)|^p
+E\sup_{t\leq T}|r^h_t|^p_{l_{h,2}}
\leq N(E|\psi|^p_{\frak m}+E\cK_{\frak m}^p(T))
\end{equation} 
holds for any $p>0$, for sufficiently large $\frak m$, 
with a constant $N$ independent of $h$.

\begin{theorem}                          \label{theorem main} 
Let $k\geq0$ be an integer, 
and let Assumptions 
\ref{assumption 1.12.10.12} through  
\ref{assumption 3.7.5.12} 
hold with 
$$
\frak m>2k+3+\frac{d}{2}. 
$$
Then there are 
continuous random fields $u^{(1)}$, $u^{(2)}$,..., $u^{(k)}$ 
on $H_T$, independent of $h$, 
such for each $h>0$ \eqref{expansion} 
holds almost surely 
for $t\in[0,T]$ and $x\in\bG_h$ with a continuous 
random field 
$(r^h_t(x))_{(t,x)\in H_T}$ satisfying \eqref{r} with 
$N=N(K,T,k, d,|\Lambda_0|)$. If $k$ is odd and 
$\frak p^{\lambda}=\frak q^{\lambda}=0$ for all $\lambda\in\Lambda_0$, 
then it is sufficient that Assumptions 
\ref{assumption 1.12.10.12} through  
\ref{assumption 3.7.5.12} hold only 
with 
$$
\frak m> 2k+2+\frac{d}{2}
$$
to have the same conclusion. Moreover, in this case 
\eqref{expansion} and \eqref{r} hold for each $h\neq0$, and $u^{(l)}=0$ 
for all odd $l\leq k$. 
\end{theorem} 

This is our main theorem which is proved 
in Section \ref{section proof}. 
It clearly implies the following result. 

\begin{corollary}                      \label{corollary 1.23.10.12}
Let $p>0$. If Assumptions 
\ref{assumption 1.12.10.12} through  
\ref{assumption 3.7.5.12} 
hold with 
$
\frak m>3+({d}/{2}),  
$
then for all $h>0$
\begin{equation*}                          
E\sup_{[0,T]\times\bG_h}|u-u^h|^p
+E\sup_{t\in[0,T]}|u_t-u_t^h|^p_{l_{h,2}} 
\leq Nh^{p}(E|\psi|^p_{\frak m}+E\cK^p_{\frak m}(T)). 
\end{equation*}
If Assumptions 
\ref{assumption 1.12.10.12} through  
\ref{assumption 3.7.5.12} 
hold with 
$
\frak m>4+({d}/{2}),  
$
and  $\frak p^{\lambda}=\frak q^{\lambda}=0$ 
for all $\lambda\in\Lambda_0$,  
then for all $h>0$
\begin{equation*}                          
E\sup_{[0,T]\times\bG_h}|u-u^h|^p
+E\sup_{t\in[0,T]}|u_t-u_t^h|^p_{l_{h,2}} 
\leq Nh^{2p}(E|\psi|^p_{\frak m}+E\cK^p_{\frak m}(T)). 
\end{equation*}
In these estimates $N$ is a constant depending only on 
$K$, $T$, $p$, $d$ and $|\Lambda_0|$. 
\end{corollary}
\begin{remark}                                                                     \label{remark sharp}
By an example given in \cite{DK} one can see that the rate of convergence 
stated in the above corollary in each of the cases is sharp. 
(See Remark 2.21 in \cite{DK}). 
\end{remark}
\begin{remark}                                                                      \label{remark monotone}If $M^{h,r}=0$ and $g^r=0$ for all $r\geq1$ 
in \eqref{scheme} then we get finite difference schemes (in the spatial variable)  
for parabolic (possibly degenerate) PDEs. Since these schemes are not necessarily 
monotone, Theorem \ref{theorem main} and 
Corollary \ref{corollary 1.23.10.12} are new results also in the case of deterministic 
PDEs. They generalise the corresponding results, Theorem 2.3 in 
\cite{GK3} and Theorems 2.16 and 2.18 in \cite{DK} on monotone schemes to 
a class of schemes which contains also non monotone finite difference schemes. 
\end{remark}
 Now we formulate some implications 
of Theorem \ref{theorem main}  
on {\it Richardson extrapolation}.   
We set  
$$
(c_0,c_1,...,c_k)=(1,1,...,1)V^{-1}, 
\quad 
(\tilde c_0,\tilde c_1,...,\tilde c_{\tilde k})
=(1,1,...,1)\tilde V^{-1}, 
$$
where $\tilde k=(k-1)/2$ for odd integer $k\geq0$, $V^{-1}$ is  
the inverse of the $(k+1)\times(k+1)$ Vandermonde matrix $V=(V^{ij})$ 
given 
by $V^{ij}=2^{-(i-1)(j-1)}$, and 
$\tilde V^{-1}$ is the inverse of the $\tilde k\times\tilde k$ Vandermonde 
matrix $\tilde V$ given by $\tilde V^{ij}=4^{-(i-1)(j-1)}$. 
Define 
$$
v^h=\sum_{i=0}^{k}c_iu^{h/2^i}
$$
when Assumptions \ref{assumption 1.12.10.12} through  
\ref{assumption 3.7.5.12} hold with $\frak m>2k+3+d/2$ 
for some integer $k\geq0$, 
and define also 
$$
\tilde v^h=\sum_{i=0}^{\tilde k}\tilde c_iu^{h/2^i}
$$
when $k\geq1$ is odd,  
$\frak p^{\lambda}=\frak q^{\lambda}=0$ for all $\lambda\in\Lambda_0$, 
and Assumptions \ref{assumption 1.12.10.12} through  
\ref{assumption 3.7.5.12} hold with $\frak m>2k+2+d/2$. 

\begin{theorem}                                 \label{theorem 2.15.10.12}
 Let $p>0$. Let Assumptions \ref{assumption 1.12.10.12} through  
\ref{assumption 3.7.5.12} hold with $\frak m>2k+3+d/2$ for 
some integer $k\geq0$. Then for all $h>0$
$$
E\sup_{t\in[0,T]}\sup_{x\in\bG_h}|u_t(x)-v_t^h(x)|^p
+E\sup_{t\in[0,T]}|u_t-v_t^h|^p_{l_{h,2}} 
$$
\begin{equation}                          \label{4.15.10.12}
\leq Nh^{p(k+1)}(E|\psi|^p_{\frak m}+E\cK^p_{\frak m}(T))
\end{equation}
holds  with a constant $N=N(K,T,k,p,|\Lambda|)$. 
Let $k\geq1$ be an odd number an let 
Assumptions \ref{assumption 1.12.10.12} through  
\ref{assumption 3.7.5.12} hold with $\frak m>2k+2+d/2$ 
Then estimate \eqref{4.15.10.12} holds for all $h>0$ with $\tilde v^h$ 
in place of $v^h$. 
\end{theorem}
 
\begin{proof}
By definition of the coefficients 
$c_i$ and $\tilde c_i$, from Theorem \ref{theorem main} 
we have that for each $h$ almost surely 
$$
u_t(x)-v_t^h(x)
=h^{k+1}R^{(h)}_t(x), 
\quad 
u_t(x)-\tilde v_t^h(x)
=h^{k+1}\tilde R^{(h)}_t(x)
$$
for all $t\in[0,T]$ and $x\in\bG_h$, where 
$R^{(h)}=\sum_{j=0}^{k}c_j2^{-j}r^{h/2^j}$ and 
$\tilde R^{(h)}=\sum_{j=0}^{\tilde k}\tilde c_j2^{-j}r^{h/2^j}$. 
Hence the theorem follows from Theorem \ref{theorem main} 
by virtue of estimate \eqref{r}. 
\end{proof}

\begin{example}
Assume that we have 
$d=2$, $\frak m=10$ and $\gp^{\lambda}=\gq^{\lambda}=0$ 
for every $\lambda\in\Lambda_0$. 
 Then
$$
 \tilde v^h:=\tfrac{4}{3}u^{h/2}-\tfrac{1}{3}u^h
$$ 
satisfies
$$
E\sup_{t\leq T}\sup_{x\in\bG_h}|u^{}_{t}(x)
- \tilde u^h_{t}( x)|^2
\leq N h^{ 8}(E|\psi|^2_{\frak m}+E\cK^2_{\frak m}(T)).
$$  
\end{example}

\begin{remark}
If 
in addition to the conditions of Theorem \ref{theorem 2.15.10.12} 
$\psi\geq0$, $f_t\geq0$ and $g_t^{r}=0$ are 
also satisfied for 
$t\in[0,T]$, $x\in\bR^d$ and $r\geq1$, then $u$, the solution of  
\eqref{equation}-\eqref{ini}, is nonnegative. Such situation arises, 
for example in the case of the Zakai equation in nonlinear 
filtering, where $f=0$, $g=0$ and $\psi$ is the conditional 
density of the initial value of the signal, given the initial value 
of the observation.   Notice, however 
that even in such cases, in general, $v^h$ and $\tilde v^h$ 
take negative values.  If we want our approximations 
to be also non-negative then we may 
take, for example, $(v^h)^+$ and $(\tilde v^h)^+$ 
in place of $v^h$ and $\tilde v^h$, respectively. 
Since $|z-w^+|\leq |z-w|$ for any $z\in[0,\infty)$ and 
$w\in\bR$, for $u_t(x)\geq0$ we have 
$$
|u_t(x)-(v_t^h(x))^+|\leq |u_t(x)-v_t^h(x)|, 
\quad
|u_t(x)-(\tilde v_t^h(x))^+|\leq |u_t(x)-\tilde v_t^h(x)|. 
$$
Consequently, if  $u$ is a nonnegative 
random field, then 
Theorem \ref{theorem 2.15.10.12} holds also with 
$(v^h)^+$ and $(\tilde v^h)^+$ in place of 
$v^h$ and $\tilde v^h$, respectively. 
\end{remark} 

Theorem \ref{theorem 2.15.10.12} implies 
the following results on almost sure 
rate of convergence.

\begin{theorem}                                     \label{theorem almost sure}
Let $(h_n)_{n=1}^{\infty}$ be 
a nonnegative sequence from $l_q$ for some $q\geq1$. 
Let Assumptions \ref{assumption 1.12.10.12} through  
\ref{assumption 3.7.5.12} hold with $\frak m>2k+1+d/2$ for 
some integer $k\geq0$. 
Then for each $\varepsilon>0$ 
there is a finite random variable $\xi_{\varepsilon}$ 
such that almost surely 
\begin{equation}                          \label{6.15.10.12}
\sup_{t\in[0,T]}\sup_{x\in\bG_{h}}|u_t(x)-v_t^{h}(x)|
+\sup_{t\in[0,T]}|u_t-v_t^{h}|_{l_{h,2}} 
\leq \xi_{\varepsilon}h^{k+1-\varepsilon} 
\end{equation}
holds for $h=h_n$, for integers $n\geq1$. 
Let $k\geq1$ be an odd number such that 
Assumptions \ref{assumption 1.12.10.12} through  
\ref{assumption 3.7.5.12} hold with $\frak m>2k-1+d/2$.  
Then for every $\varepsilon>0$ there is a finite random variable 
$\xi_{\varepsilon}$ such that almost surely 
\eqref{6.15.10.12} holds with $\tilde v^{h}$ 
in place of $v^{h}$ for $h=h_n$ for all $n\geq1$.  
\end{theorem}

\begin{proof}
We prove only the statement concerning $v^h$, since the assertion for 
$\tilde v^h$ can be proved in the same way.  
Set $\Omega_{r}=\{\omega\in\Omega: |\psi(\omega)|_{\frak m}|\leq r\}$ 
for integers $r\geq1$. Then $\Omega_r\in\cF_0$ and 
$\cup_{r=1}^{\infty}\Omega_r$ has full probability. If for each 
$r$  we have 
\eqref{6.15.10.12}  on  $\Omega_r$, with some almost surely 
finite random variable $\xi_{\varepsilon r}$ instead of 
$\xi_{\varepsilon}$, then clearly 
we have \eqref{6.15.10.12} almost surely 
with some finite random variable $\xi_{\varepsilon}$. 
Thus we may additionally assume that $|\psi|_{\frak m}$ 
is bounded by a constant. Define 
$$
\tau_r=\inf\{t\geq0:\cK_{\frak m}(t)\geq r\}\quad 
\text{for integers $r\geq1$}. 
$$
Then $\tau_r$ is a stopping time for each $r$, and due to 
Assumption \ref{assumption 2.15.9.12}, 
$\tau_r\to\infty$ as $r\to\infty$. 
Thus if for each $r$ we have a 
finite random variable $\xi_{\varepsilon r}$ 
such that \eqref{6.15.10.12} holds almost surely 
with $u_{t\wedge\tau_r}$, $v_{t\wedge\tau_r}^{h}$ and 
$\xi_{\varepsilon r}$ in place of 
$u_{t}$, $v_{t}^{h}$ and 
$\xi_{\varepsilon}$, respectively for every $h=h_n$, 
then there is a finite random variable $\xi_{\varepsilon}$ 
such that  \eqref{6.15.10.12} holds almost surely for 
all $h=h_n$, $n\geq1$.  
Therefore in addition to the assumptions of the theorem 
we may assume that $|\psi|_{\frak m}+\cK_{\frak m}(T)$ 
is bounded by a constant, say $c$. 
Let $\eta_h$ denote the left-hand side of inequality 
\eqref{6.15.10.12}, and set 
$\zeta_n=h_n^{-k-1+\varepsilon}\eta_{h_n}$.  
Then under the additional assumption that 
$|\Psi|_{\frak m}+\cK_{\frak m}(T)\leq c$,  
by Theorem \ref{theorem 2.15.10.12} we obtain 
$$
E\zeta_n^p\leq Nc^ph^{\varepsilon p}_n
\quad
\text{for all $n\geq1$ and $p>0$},  
$$
where the constants $N$ and $c$ are independent of $n$. 
Taking here $p$ so large that $p\varepsilon\geq q$, 
we get  
$$
E\sum_{n=1}^{\infty}\zeta_n^p
=\sum_{n=1}^{\infty}E\zeta_n^p
\leq Nc^p\sum_{n=1}^{\infty}h^{\varepsilon p}_n<\infty. 
$$
Hence for 
$$
\xi_{\varepsilon}:=\left(\sum_{n=1}^{\infty}\zeta_n^p\right)^{1/p}
$$
we obtain that almost surely $\xi_{\varepsilon}<\infty$ and 
$\eta_{h_n}\leq\xi_{\varepsilon}h^{k+1-\varepsilon}_n$ for 
all $n$, which completes the proof of the theorem. 
\end{proof}

\begin{example}
Consider the degenerate parabolic SPDE 
$$
du_{t}=2D^{2}u_{t}\,dt+2Du_{t}\,dw_{t}\,\quad  
t\in(0,1], \,x\in\bR, 
$$
with initial condition $u_{0}(x)=\cos x$, 
and approximate it by the finite difference equation 
$$
du^{h}_{t}(x)
=\frac{u^{h}_{t}(x+2h)-2u^{h}_{t}(x)+u^{h}_{t}(x-2h)}
{2h^{2}}\,dt+\frac{u^{h} _{t}(x+h)-u^{h}_{t}(x-h)}{h}\,dw_{t}. 
$$
The unique bounded
solution of the SPDE problem  is 
$$
u_{t}(x)=\cos (x+2w_{t}),  
$$ 
and the unique bounded solution to the 
finite difference equation 
(with initial condition $u^h_{0}(x)=\cos x$) 
is
$$
u^{h}_{t}(x)=\cos(x+2\phi_{h}w_{t}), 
$$
where 
$
\phi_{h}=\sin h/h.
$
For $t=1$, $h=0 .1$, and $w_{t}=1$ we have
$$
u_{1}(0)\approx-0.4161468365,
$$
$$u_{1}^{h}(0)\approx-0.4131150562
,\quad 
u_{1}^{h/2}(0)\approx-0.415389039,
$$
$$
 \tilde u^h_{1}(0)=\tfrac{4}{3}u^{h/2}_{1}(0)
-\tfrac{1}{3}u^h_{1}(0)=\approx0.4161470333 	.
$$ 
Such level of
accuracy by $u^{\tilde h}_{1}(0)$ is achieved  with 
$\tilde h=0.0008$,
which is more than 60 times smaller than $h/2$. 

Note that since $\cos x$  
is not square integrable over $\bR$, the above example 
does not  fit into our setting, but  using 
suitable Sobolev spaces we can extend our setting 
so that this example can be included. 
\end{example}

\mysection{Preliminaries}                   \label{section 2.2.5.12}

We introduce some notation. 
For $\lambda\in\Lambda_1\cup\{-\Lambda_1\}$  
and 
for $h\neq0$ we 
use 
$T_{h,\lambda}$ and $T^h_{\lambda}$ for the operators 
defined by  
$$
T_{h,\lambda}\varphi(x)=\varphi(x+h\lambda),  
\quad 
T^h_{\lambda}=\frac{1}{2}(T_{h,\lambda}+T_{h,-\lambda}) 
$$
for functions $\varphi$ given on $\bR^d$, where $I$ 
denotes the identity operator.  
For the sake of uniformity of notation 
we take a set of unit vectors, 
$\Lambda_2=\{\ell_1,....,\ell_d\}$,  
which is disjoint from $\Lambda_0\cup\{-\Lambda_0\}$,   
and use the notation 
$$
\delta_{h,l_i}=\delta^h_{\ell_i}=D_i=\frac{\partial}{\partial x^i}
$$
for the partial derivative in the direction of 
the $i$-th basis vector $e_i$ 
for $i=1,...,d$ and $h\neq0$. 
 We use also the notation $I_{\lambda}^h$ 
for $T^h_{\lambda}$ when $\lambda\in\Lambda_1$ and 
for the identity when $\lambda\in\Lambda_2$. 
Set $T_{h,\lambda}=I$ for 
$\lambda\in\Lambda_2$. It is easy to see that 
\begin{align}                                                        
\delta_{h,\lambda}(uv)=&v\delta_{h,\lambda}u
+(T_{h,\lambda}u)\delta_{h,\lambda}v                      \label{1.3.5.12} \\
=&
v\delta_{h,\lambda}u
+u\delta_{h,\lambda}v
+h_{\lambda}(\delta_{h,\lambda}u)(\delta_{h,\lambda}v)     \label{2.3.5.12} 
\end{align} 
for all $h\neq0$ and 
$$
\lambda\in\Lambda:=\Lambda_0\cup\Lambda_2,  
$$
where $h_{\lambda}=h$ if $\lambda\in\Lambda_0$ 
and $h_{\lambda}=0$ if $\lambda\in\Lambda_2$. 
Hence we get  
\begin{equation}                                          
                                       \label{8.7.5.12}
\delta_{\lambda}^h(uv)
=(\delta_{\lambda}^hu) I_{\lambda}^hv+
(I_{\lambda}^h u)\delta_{\lambda}^hv
\quad\text{for $\lambda\in\Lambda$.}
\end{equation}
Indeed, this is the classical Leibniz rule 
when $\lambda\in\Lambda_2$, and for 
$\lambda\in\Lambda_0$ 
by virtue of \eqref{1.3.5.12} 
we have 
\begin{align*}
\delta_{\lambda}^h(uv)
&=(\delta_{\lambda}^hu)v+\frac{1}{2}
\{(\delta_{h,\lambda}v)T_{h,\lambda}u
+(\delta_{-h,\lambda}v)T_{h,-\lambda}u \} \\       
&=(\delta_{\lambda}^hu)v+
(\delta^{h}_{\lambda}v)T_{-h,\lambda}u+
\frac{1}{2}
(T_{h,\lambda}u-T_{-h,\lambda}u)\delta_{h,\lambda}v\\
&=(\delta_{\lambda}^hu)v+
(\delta^{h}_{\lambda}v)T_{-h,\lambda}u+
(\delta^h_{\lambda}u)(T_{h,\lambda}v-v) \\
&=(\delta_{\lambda}^hu)T_{h,\lambda}v+
(\delta^{h}_{\lambda}v)T_{-h,\lambda}u.  
\end{align*}
Since $\delta_{\lambda}^h=\delta_{\lambda}^{-h}$, 
symmetrizing 
the last equality in $h$ we get \eqref{8.7.5.12}.

To extend this Leibniz rule 
to higher order finite differences and derivatives 
we introduce further notation. 

Let $\Lambda^n$ denote the set 
of sequences $\lambda_1....\lambda_n$ 
of length $n$  
of elements of $\Lambda$. 
We use the notation $|\alpha|:=n$ for the length 
of $\alpha\in\Lambda^n$. We introduce a `sequence' of `length zero', 
which we denote by $\epsilon$.  
It is considered a sub-sequence of 
any $\lambda\in\Lambda^n$ 
for each $n\geq1$.  For 
$\mu\in\Lambda^m$ and $\lambda\in\Lambda^n$, $m\leq n$ 
we write $\mu\leq\lambda$ if $\mu$ is a sub-sequence 
of $\lambda$, 
and $\lambda\setminus\mu$ denotes the sequence we obtain from 
$\lambda$ by removing $\mu$ from it. In particular, 
$\lambda\setminus\lambda=\epsilon$, $\lambda\setminus\epsilon=\lambda$.  
For $\lambda=\lambda_1....\lambda_n\in\Lambda^n$ we set  
$$
\delta_{\lambda}=\delta_{\lambda}^h
=\delta^h_{\lambda_1}\dots\delta^h_{\lambda_n}, 
\quad 
I_{\lambda}=I_{\lambda}^h
=I_{\lambda_1}^h...I_{\lambda_n}^h, 
$$
and for $\epsilon$ we set 
$$
\delta_{\varepsilon}=I_{\epsilon}=I. 
$$ 
Now a generalization of \eqref{1.3.5.12} reads as follows.

\begin{lemma}                                     \label{lemma 1.3.5.12}
Let $\lambda\in\Lambda^n$ for $n\geq1$. Then 
$$
\delta_{\lambda}(uv)
=\sum_{\mu\leq\lambda}(\delta_{\mu}
I_{\lambda\setminus\mu}u)
(\delta_{\lambda\setminus\mu}I_{\mu}v), 
$$
where the summation is taken over all sub-sequences of $\lambda$, 
including 
$\epsilon$. 
\end{lemma}

\begin{proof} 
The lemma can be proved by a straightforward  
induction on $n$, the length of the multi-sequence $\lambda$. 
\end{proof}
Notice that for $\lambda\in\Lambda$ 
\begin{equation}                                \label{1.8.5.12}
I_{\lambda}=I+\tfrac{h^2}{2}\Delta_{\lambda}^h=hP_{\lambda}+I, 
\end{equation}
where 
$$
\Delta^h_{\lambda}
:=
\frac{1}{h^2}(T_{h,\lambda}-2I+T_{h,-\lambda})
=\delta_{h,\lambda}\delta_{h,-\lambda}
=(\delta^{h/2}_{\lambda})^2
=\frac{1}{h}(\delta_{h,\lambda}-\delta_{-h,\lambda}), 
$$
$$
P_{\lambda}:=\tfrac{1}{2}(\delta_{h,\lambda}-\delta_{-h,\lambda})
$$ 
for $\lambda\in\Lambda_1$, 
and $\Delta_{\lambda}^h=P_{\lambda}=0$ for $\lambda\in\Lambda_2$ 
and $h\neq0$.   
Hence for $\alpha=\alpha_1\alpha_2...\alpha_m\in\Lambda^m$, 
$m\geq1$ we get 
\begin{equation}                                                      \label{2.8.5.12}                                                                 
I_{\alpha}
=I+{h^2}\mathcal O_{\alpha}=I+h\mathcal P_{\alpha}, 
\end{equation}
where 
$$
\mathcal O_{\alpha}=(\Delta_{\alpha_1}I_{\alpha_2\alpha_3....\alpha_m}
+\Delta_{\alpha_2}I_{\alpha_3....\alpha_m}+
...+\Delta_{\alpha_{m-1}}I_{\alpha_m}
+\Delta_{\alpha_{m}})/2,    
$$
$$
\mathcal P_{\alpha}=P_{\alpha_1}I_{\alpha_2\alpha_3....\alpha_m}
+P_{\alpha_2}I_{\alpha_3....\alpha_m}+
...+P_{\alpha_{m-1}}I_{\alpha_m}
+P_{\alpha_{m}}). 
$$
Set 
$$
R_{\lambda}=(T_{h,\lambda}-T_{h,-\lambda})/2
\quad
\text{for $\lambda\in\Lambda$,} 
$$
and notice that for any functions $a$ and $u$ 
on $\bR^d$ we have 
\begin{align}
I_{\mu}(au)=&aI_{\mu}u                               
+h(P_{\mu}a)I_{\mu}u+h(\delta_{\mu}a)(R_{\mu}u)    \label{3.18.9.12}\\
=&(I_{\mu}a)(I_{\mu}u)+(R_{\mu}a)(R_{\mu}u)       \label{4.18.9.12}
\end{align}
for all $\mu\in\Lambda$ and $h\neq0$. 
Indeed, if $\mu\in\Lambda_2$, then 
$P_{\mu}=R_{\mu}=0$, and 
these equalities hold by definition. 
If $\mu\in\Lambda_0$ then 
\begin{align*}
I_{\mu}(au)
=&aI_{\mu}u
+\frac{h}{2}
\{(\delta_{h,\mu}a)T_{h,\mu}u+
(\delta_{h,-\mu}a)T_{h,-\mu}u\}\\
=&aI_{\mu}u+h(P_{\mu}a)(T_{h,-\mu}u)+(R_{\mu}a)(R_{\mu}u), 
\end{align*}
and symmetryzing the right-hand side of the last 
equality in $\mu$ and $-\mu$,  
we obtain \eqref{3.18.9.12}. Hence we get 
equality \eqref{4.18.9.12} by using 
\eqref{1.8.5.12}. 

We will often make use of the fact that 
for $\lambda\in\Lambda^k$ 
$$
|\delta_{\lambda}v|_0\leq N|v|_k, 
\quad\text{for $v\in H^k$}, 
$$
where $N$ is a constant depending 
only on $d$, $k$ and $|\Lambda|$. 

More precisely, the following lemma holds. 
\begin{lemma}                                      \label{lemma 1.5.10.12}
Let $k$ be a positive integer. 
Then for 
$\lambda=\lambda_1\lambda_2...\lambda_k\in\Lambda^k$ and $h\neq0$ 
$$
|\delta_{h,\lambda}v|_0\leq \Pi_{i=1}^k|\lambda_i||D^kv|_0, 
\quad
|\delta_{\lambda}^hv|_0\leq \Pi_{i=1}^k|\lambda_i||D^kv|_0
$$
for all $v\in H^k$. 
\end{lemma}
\begin{proof}
The second inequality clearly follows from the first one. 
It is sufficient to prove the first inequality for smooth functions $v$ 
with compact support. For $\lambda\in\Lambda$ let 
$\bar\lambda$ denote $\lambda$ if $\lambda\in\Lambda_0$, and let 
$\bar\lambda=e_i$ if $\lambda=l_i\in\Lambda_2$. Then 
$$
\delta_{h,\lambda}v(x)=\int_{0}^{1}
\partial_{\bar\lambda}v(x+h_{\lambda}\theta\lambda)\,d\theta 
$$
for every $x\in\bR^d$, 
where $\partial_{\bar\lambda}$ 
is the directional 
derivative along 
$\bar\lambda$, $h_{\lambda}=h$ 
for $\lambda\in\Lambda_0$ and $h_{\lambda}=0$ 
for $\lambda\in\Lambda_2$. 
Hence for $\lambda\in\Lambda^k$ and smooth $v$ we get 
by induction on $k$ that 
$$
\delta_{h,\lambda}v(x)
=\int_{[0,1]^k}
\partial_{\bar\lambda_1}...\partial_{\bar\lambda_k}
v(x+\theta_1h_{\lambda_1}\lambda_1
+...+\theta_kh_{\lambda_k}\lambda_k)
\,d\theta_1\,d\theta_2...d\theta_k
$$
for every $x\in\bR^d$, which by Minkowski's 
inequality and by the shift 
invariance of the Lebesgue measure yields 
$$
|\delta_{h,\lambda}v|_0
\leq |\partial_{\bar\lambda_1}...\partial_{\bar\lambda_k}v|_0. 
$$
We can finish the proof by noting that 
$$
|\partial_{\bar\lambda_1}...\partial_{\bar\lambda_k}v(x)|^2
\leq |D^kv(x)|^2\Pi_{i=1}^k|\lambda_i|^2.  
$$
\end{proof}
\begin{lemma}                                    \label{lemma 6.7.5.12}
Let $\fra$ be a bounded  
function on $\bR^d$. Assume that the first order 
partial derivatives of $\fra$ are functions, which 
together with $a$ are bounded in magnitude 
by a constant $K$. 
Then for all $h\neq0$ and $u\in H^1$ we have 
\begin{align}
(I_{\mu}u,\fra\delta_{\lambda}u)
=&-\frac{1}{2}((\delta_{\lambda}\fra)I_{\lambda}I_{\mu}u,u)
-\frac{1}{2}(I_{\mu}R_{\lambda}u,(P_{\lambda}\fra)u)         \nonumber\\
&-\frac{1}{2}(R_{\lambda}u,(P_{\mu}\fra)I_{\mu}u)                
-\frac{1}{2}(R_{\lambda}u,(\delta_{\mu}\fra)R_{\mu}u), \label{1.18.9.12}\\
|(I_{\mu}u,\fra\delta_{\lambda}u)|\leq& N|u|_0^2       \label{2.18.9.12}
\end{align}
for $\lambda,\mu\in\Lambda$, 
where $N$ is a constant depending only 
on $K$, $d$, $|\lambda|$ and $|\mu|$. 
\end{lemma}

\begin{proof} 
Notice that 
$$
\delta_{\lambda}^{\ast}=-\delta_{\lambda},  
\quad I_{\lambda}^{\ast}=I_{\lambda}
\quad
\text{for $\lambda\in\Lambda$} 
$$ 
for the adjoints $\delta_{\lambda}^{\ast}$ and 
$I_{\lambda}^{\ast}$ of 
$\delta_{\lambda}$ and 
$I_{\lambda}$  in $L_2(\bR^d)$, respectively. 
Using this, the Leibniz rule \eqref{8.7.5.12}, 
and taking into account 
that $\delta_{\lambda}$ and $I_{\mu}$ commute, 
we have 
$$
(I_{\mu}u,\fra\delta_{\lambda}u)=
-(\delta_{\lambda}(\fra I_{\mu}u),u)
=-((\delta_{\lambda}\fra)I_{\lambda}I_{\mu}u,u)-A
$$
with 
$$
A=((I_{\lambda}\fra)\delta_{\lambda}I_{\mu}u,u)
=(\delta_{\lambda}u,I_{\mu}((I_{\lambda}\fra)u)). 
$$  
Using \eqref{1.8.5.12} we get 
$$
A=h(\delta_{\lambda}u,I_{\mu}((P_{\lambda}\fra)u)+B
=(I_{\mu}R_{\lambda}u,((P_{\lambda}\fra)u)+B
$$
with
$$
B=(\delta_{\lambda}u,I_{\mu}(\fra u)),  
$$
and using \eqref{3.18.9.12} we obtain 
\begin{align*}
B=&(\delta_{\lambda}u,\fra I_{\mu}u)
+h(\delta_{\lambda}u,(P_{\mu}\fra)I_{\mu}u)                
+h(\delta_{\lambda}u,(\delta_{\mu}\fra)R_{\mu}u)\\
=&(\delta_{\lambda}u,\fra I_{\mu}u)
+(R_{\lambda}u,(P_{\mu}\fra)I_{\mu}u)                
+(R_{\lambda}u,(\delta_{\mu}\fra)R_{\mu}u). 
\end{align*}
Hence 
\begin{align}
(I_{\mu},\fra\delta_{\lambda}u)
=&-((\delta_{\lambda}\fra)I_{\lambda}I_{\mu}u,u)
-(I_{\mu}R_{\lambda}u,(P_{\lambda}\fra)u)\nonumber\\
&-(I_{\mu}u,\fra\delta_{\lambda}u)
-(R_{\lambda}u,(P_{\mu}\fra)I_{\mu}u))                
-(R_{\lambda}u,(\delta_{\mu}\fra)R_{\mu}u), 
\end{align}
which gives \eqref{1.18.9.12}. To prove 
\eqref{2.18.9.12} notice that 
$$
|P_{\lambda}\fra|\leq |\lambda||D\fra|, 
\quad 
|R_{\lambda}\fra|\leq |\lambda||D\fra|, 
$$
where $|D\fra|$ is the length of the gradient of $\fra$. 
Notice also 
that due to the shift invariance  
of the Lebesgue measure, the linear operators 
$I_{\mu}$ and $R_{\lambda}$ are contractions 
on $H^0$. Hence it is easy to see that \eqref{1.18.9.12} 
implies \eqref{2.18.9.12}.  
\end{proof}
\begin{corollary}                         \label{corollary 1.6.10.12}
Let $m$ be a nonnegative integer and 
let $b=(b^{r}(x))_{r=1}^{\infty}$ be an $l_2$-valued 
function on $\bR^d$ such that $b$ and its 
partial derivatives up to order $\max(m,1)$ are 
$l_2$-valued functions 
which in magnitude are 
bounded by $K$. Then for all $h\neq0$, $u\in H^{m}$ 
and each $\alpha\in\Lambda^m$ 
\begin{equation}                                \label{2.6.10.12}
\sum_{r=1}^{\infty}|(\delta_{\alpha}
I_{\mu}u,\delta_{\alpha}(b^{r}\delta_{\lambda}u))|^2_0
\leq N|u|_m^4       
\end{equation}
for $\lambda\in\Lambda_0$, $\mu\in\Lambda$, and $\alpha\in\Lambda^m$, 
where $N$ is a constant depending only 
on $K$, $d$, $m$ and $|\Lambda|$. 
\end{corollary}
\begin{proof}
Consider the case $m=0$. Then by definition $\delta_{\alpha}$ 
is the identity, and by \eqref{1.18.9.12}
we have 
$$
\sum_{r=1}^{\infty}
|(I_{\mu}u,b^{r}\delta_{\lambda}u)|^2_0
\leq 4\sum_{i=1}^4A_i 
$$
with
$$
A_1=\sum_{r=1}^{\infty}
|(u\delta_{\lambda}b^{r},I_{\lambda}I_{\mu}u)|^2_0
\leq 
 |u|^2_0\sum_{r=1}^{\infty}
|u\delta_{\lambda}b^{r}|^2_0
\leq N|u|^4_0, 
$$
$$
A_2=\sum_{r=1}^{\infty}
|(I_{\mu}R_{\lambda}u,uP_{\lambda}b^{r})|^2
\leq |u|^2_0\sum_{r=1}^{\infty}
|uP_{\lambda}b^{r}|^2_0
\leq N|u|^4_0,
$$
$$   
A_3=\sum_{r=1}^{\infty}
|(R_{\lambda}u,(P_{\mu}b^{r})I_{\mu}u)|^2_0
\leq
|u|^2_0\sum_{r=1}^{\infty}
|(I_{\mu}u)P_{\mu}b^{r}|^2_0
\leq N|u|^4_0,
$$
$$
A_4=\sum_{r=1}^{\infty}
|(R_{\lambda}u,(\delta_{\mu}b^{r})R_{\mu}u)|
\leq 
|u|^2_0\sum_{r=1}^{\infty}
|(R_{\mu}u)\delta_{\mu}b^{r}|^2_0
\leq N|u|^4_0,
$$
where $N$ is a constant depending only on 
$m$, $d$, $K$ and $|\Lambda|$. 
This proves \eqref{2.6.10.12} when $m=0$. 
Assume now $m\geq1$. Then 
$$
\sum_{r=1}^{\infty}|(\delta_{\alpha}
I_{\mu}u,\delta_{\alpha}(b^{r}\delta_{\lambda}u))|^2_0
\leq 2B_1+2B_2, 
$$
where 
$$
B_1=\sum_{r=1}^{\infty}
|(I_{\mu}u_{\alpha},
(I_{\alpha}b^{r})\delta_{\lambda}u_{\alpha})|^2_0,
$$
$$
B_2=\sum_{r=1}^{\infty}
|(I_{\mu}u_{\alpha}, 
\delta_{\alpha}(b^{r}\delta_{\lambda}u)
-(I_{\alpha}b^{r})\delta_{\lambda}u_{\alpha})|^2_0. 
$$
Using that \eqref{2.6.10.12} holds when $m=0$, we have 
$$
B_1\leq N|\delta_{\alpha}u|^4_0\leq N'|u|^4_{m}  
$$
with a constant $N'=N'(m,d,K,|\Lambda|)$. 
Using Lemma \ref{lemma 1.3.5.12}
and Lemma \ref{lemma 1.5.10.12}
we get 
$$
B_2\leq |u_{\alpha}|^2_0
\sum_{\rho=1}^{\infty}
|\delta_{\alpha}(\frb^{\rho}\delta_{\lambda}u)
-(I_{\alpha}\frb^{\rho})\delta_{\lambda}u_{\alpha})|^2
\leq N|u|^4_m
$$
with a constant $N=N(m,d,K,|\Lambda|)$, 
and the proof is complete.  
\end{proof}

\begin{lemma}                                     \label{lemma 1.4.10.12}
Let $m\geq0$ be an integer and let 
$\frak p$ be a nonnegative function on $\bR^d$ 
such that its partial derivatives up to order 
$\max(m,1)$ are functions, which together with 
$\frak p$ are bounded in magnitude by  
a constant $K$. Then for every $\alpha\in\Lambda^m$, 
$\lambda\in\Lambda_0$ and for all $h\in(0,\infty)$ 
we have 
$$
(\delta_{\alpha}u,\delta_{\alpha}(\frak p\delta_{h,\lambda}u))\leq N|u|^2_m
$$
for all $u\in H^{m}$, where $N$ is a constant depending only on 
$K$, $d$, $|\Lambda|$ and $m$. 
\end{lemma} 

\begin{proof}
Clearly 
$$
\delta_{\alpha}(\frak p\delta_{h,\lambda}u)
=\delta_{\alpha}(\frak p\delta_{h,\lambda}u)
-(I_{\alpha}\frak p)\delta_{\alpha}\delta_{h,\lambda}u
+(I_{\alpha}\frak p)\delta_{h,\lambda}\delta_{\alpha}u,
$$
and 
$$
(\delta_{\alpha}u,\delta_{\alpha}(\frak p\delta_{\lambda}u)
-(I_{\alpha}\frak p)\delta_{\alpha}\delta_{\lambda}u)\leq N|u|^2_m
$$
with a constant $N=N(m,K,d,\Lambda)$. By the Leibniz rule 
\eqref{2.3.5.12} we have 
$$
(\delta_{\alpha}u)\delta_{h,\lambda}\delta_{\alpha}u
=\tfrac{1}{2}\delta_{h,\lambda}(\delta_{\alpha}u)^2
-h_{\lambda}(\delta_{h,\lambda}\delta_{\alpha}u)^2. 
$$
Consequently,
$$
(\delta_{\alpha}u,\delta_{\alpha}(\frak p\delta_{h,\lambda}u))
\leq N|u|^2_m
+(I_{\alpha}\frak p,
(\delta_{\alpha}u)\delta_{h,\lambda}\delta_{\alpha}u)
$$
\begin{equation}                                \label{1.4.10.12}
= N|u|^2_m
+
\frac{1}{2}
(I_{\alpha}\frak p,\delta_{h,\lambda}(\delta_{\alpha}u)^2)
-h_{\lambda}
(I_{\alpha}\frak p,(\delta_{h,\lambda}\delta_{\alpha}u)^2). 
\end{equation}
Due to $I_{\alpha}\frak p\geq0$ and $h_{\lambda}\geq0$ 
we have 
$$
h_{\lambda}
(I_{\alpha}\frak p,(\delta_{h,\lambda}\delta_{\alpha}u)^2)\geq0, 
$$
and by taking the adjoint 
$\delta_{h,\lambda}^{\ast}=\delta_{h,-\lambda}$ 
in $L_2$, we have 
$$
(I_{\alpha}\frak p,\delta_{h,\lambda}(\delta_{\alpha}u)^2)
=
(\delta_{h,-\lambda}I_{\alpha}\frak p,(\delta_{\alpha}u)^2)
\leq N|u|^2_m
$$
with a constant $N=N(K,m,d,|\Lambda|)$. 
Thus the lemma follows from \eqref{1.4.10.12}. 
\end{proof}

\begin{lemma}                           \label{lemma weakerror}
The following 
statements hold for all $h>0$ and $\lambda,\mu\in\Lambda_0$. 
\begin{itemize}
\item[(i)] For $v,\varphi\in H^1$ we have 
\begin{equation}                        \label{1.29.1012}
|(\delta_{\lambda}^hv,\varphi)
-(\partial_{\lambda}v,\varphi)|\leq
\frac{h}{2}|\lambda|^2|Dv|_0|D\varphi|_0.
\end{equation}
\item[(ii)] Let $b=(b^{r}(x))_{r=1}^{\infty}$ be 
an $l_2$-valued function on $\bR^d$, with $l_2$ norm 
bounded by a constant $K$. Assume that its first 
order derivatives in $x$ are $l_2$-valued functions, 
which are in $l_2$ norm bounded by $K$.   
Then there is a constant $N=N(K,d)$ such that 
for $v,\varphi\in H^1$ 
$$
\sum_{r=1}^{\infty}|(b^r\delta_{\lambda}^hv,\varphi)
-(b^r\partial_{\lambda}v,\varphi)|^2\leq
Nh^2|\lambda|^4|Dv|_0^2|\varphi|_1^2. 
$$
\item[(iii)] Let $a=a(x)$ be a real function on $\bR^d$, 
bounded by a constant $K$, 
such that its derivatives up to second order are functions, 
bounded by $K$. Then for all $v\in H^1$ and $\varphi\in H^2$ 
we have  
\begin{align}
|(a\delta^h_{\mu}\delta^h_{\lambda}v,\varphi)
-(\partial_{\lambda}v,\partial_{-\mu}(a\varphi))|
\leq& \frac{h}{2}
|\lambda|(|\lambda|+|\mu|)|Dv|_0|D^2(a\varphi)|_0      \nonumber\\
\leq& 
Nh|Dv|_0|\varphi|_2                               \label{3.29.10.12}
\end{align}
with a constant $N=N(K,d, |\lambda|,|\mu|)$. 
\end{itemize}
\end{lemma}

\begin{proof}
It is sufficient to prove the lemma 
for $v,\varphi\in C_0^{\infty}(\bR^d)$. For such 
$v$ and $\varphi$ 
we have   
$$
(\delta^h_{\lambda}v-\partial_{\lambda}v,\varphi)
=\frac{h}{4}
\int_{-1}^1\int_{-1}^1
\int_{\bR^d}\theta_1
\partial^2_{\lambda}v(x+\theta_1\theta_2h\lambda)\varphi(x)
\,dx\,d\theta_1\,d\theta_2. 
$$
Hence by integration by parts, using 
the Bunyakovsky-Cauchy-Schwarz inequality 
and the shift invariance of the Lebesgue measure,  
we obtain 
$$
|(\delta^h_{\lambda}v-\partial_{\lambda}v,\varphi)|
\leq \frac{h}{4}\int_{-1}^1\int_{-1}^1|\theta_1|
|\partial_{\lambda}v|_0
|\partial_{\lambda}\varphi|_0
\,d\theta_1\,d\theta_2=\frac{h}{2}
|\partial_{\lambda}v|_0
|\partial_{\lambda}\varphi|_0, 
$$
which proves \eqref{1.29.1012}. Hence we can easily 
obtain (ii) by substituting 
$b^r\varphi$ in place of $\varphi$ in \eqref{1.29.1012}, 
and then by taking the square and summing up 
the inequalities over $r\in\{1,2,...\}$.  To prove (iii) 
notice that 
$$
(a\delta_{\mu}^h\delta_{\lambda}^hv,\varphi)
-(\partial_{\lambda}v,\partial_{-\mu}(a\varphi))=A+B
$$
with
$$
A=(v,\delta_{\mu}^h\delta_{\lambda}^h(a\varphi)-
(v,\partial_{\mu}\delta_{\lambda}^h(a\varphi)
$$
$$
B:=(v,\delta_{\lambda}^h\partial_{\mu}(a\varphi)
-(v,\partial_{\lambda}\partial_{\mu}(a\varphi)). 
$$
By virtue of (i) we have 
$$
|A|\leq \frac{h}{2}|\mu^2||Dv|_0|D\delta_{\lambda}(a\varphi)|_0
\leq \frac{h}{2}|\mu|^2|\lambda||Dv|_0|D^2(a\varphi)|_0,
$$
$$
|B|\leq \frac{h}{2}|\lambda|^2|Dv||DD_{\mu}(a\varphi)|_0 
\leq \frac{h}{2}|\lambda|^2|\mu||Dv|_0|D^2(a\varphi)|_0, 
$$
and by adding up these estimates we get \eqref{3.29.10.12}. 
\end{proof}

Next we formulate a generalization of 
a well-known lemma from \cite{OR}. 

\begin{lemma}                         \label{lemma 1.15.9.12}
Let $\sigma(x)$ be a $d\times m$ matrix for every 
$x\in\bR^d$ such that 
$$
|\sigma(x)-\sigma(y)|
\leq L|x-y|\quad\text{for all $x,y\in\bR^d$}.
$$
Let $V$ be a symmetric $d\times d$ matrix. 
Then the following estimates hold.
\begin{enumerate}
\item[(i)] For every $\lambda\in\bR^d\setminus\{0\}$ 
$$
|\partial_{\lambda}(\sigma\sigma^{\ast})^{ij}V^{ij}|^2
\leq 4L^2|\lambda| (\sigma\sigma^{\ast})^{ij}V^{ik}V^{jk}
\quad \text{for $dx$-almost every $x\in\bR^d$}; 
$$
\item[(ii)] for every 
$\lambda\in\Lambda_0$, $h\neq0$ , 
$\varepsilon>0$ and $x\in\bR^d$ we have 
$$
|\delta_{h,\lambda}(\sigma\sigma^{\ast})^{ij}V^{ij}|^2
\leq 4(1+\varepsilon)
L^2|\lambda| (\sigma\sigma^{\ast})^{ij}V^{ik}V^{jk}
+(1+\varepsilon^{-1})L^4|\lambda|^4h^2|V|^2;
$$
\item[(iii)] for every $\lambda\in\Lambda$, $h\neq0$  and $x\in\bR^d$ 
 we have 
$$
|\delta^h_{\lambda}(\sigma\sigma^{\ast})^{ij}V^{ij}|^2
\leq 5L^2|\lambda| (\sigma\sigma^{\ast})^{ij}V^{ik}V^{jk}
+5L^4|\lambda|^4h^2|V|^2.  
$$
\end{enumerate}
\end{lemma}

\begin{proof} 
Part (i) is well-known from \cite{OR}, and (iii) 
obviously follows from (ii). To prove (ii) we use  
the discrete Leibniz rule 
to get 
$$
\delta_{h,\lambda}(\sigma\sigma^{\ast})
=\sigma_{\lambda}\sigma^{\ast}+
\sigma\sigma^{\ast}_{\lambda}
+
h\sigma_{\lambda}\sigma^{\ast}_{\lambda}, 
$$
where $\sigma_{\lambda}:=\delta_{h,\lambda}\sigma$. 
Hence using the Cauchy inequality and the simple inequality 
$$
(a+b)^2\leq (1+\varepsilon)a^2+(1+\varepsilon^{-1})b^2, 
$$
we have 
$$
|\delta_{h,\lambda}(\sigma\sigma^{\ast})^{ij}V^{ij}|^2
=|2\sigma_{\lambda}^{ik}\sigma^{jk}V^{ij}
+h\sigma_{\lambda}^{ik}\sigma_{\lambda}^{jk}V^{ij}|^2
$$
$$
\leq 4(1+\varepsilon)|\sigma_{\lambda}|^2
(\sigma\sigma^{\ast})^{jl}V^{ji}V^{li}
+(1+\varepsilon^{-1})h^2\|\sigma_{\lambda}\|^4 |V|^2,  
$$
and obtain (ii) by taking into account 
$|\sigma_{\lambda}|\leq L\lambda$.
\end{proof}
Finally we present a stochastic Gronwall  
lemma from \cite{G12}, which improves Lemma 3.7 
from \cite{GS}. 
 
\begin{lemma}                             \label{lemma 1.6.10.12}
Let $y=(y_t)_{t\in[0,T]}$ 
and $F=(F_t)_{t\in[0,T]}$ 
be  
adapted nonnegative stochastic processes 
and let $m=(m_t)_{t\in[0,T]}$ be a 
continuous local martingale  
such that almost surely 
$$
dy_t\leq N(y_t+F_t)\,dt+dm_t
\quad\text{for all $t\in[0,T]$}
$$
where $N$ is a constant. 
Assume that 
$$
d\<m\>_t\leq N (y^2_t+G_ty_t)\,dt
$$
for a nonnegative stochastic process 
$G=(G_t)_{t\in[0,T]}$ 
and a constant $N$. Then for every $p>0$ 
$$
E\sup_{t\leq T}|y_t|^p\leq CE|y_0|^p
+CE\left\{\int_0^{T}(F_t+G_t)\,dt\right\}^{p}, 
$$
where $C$ is a constant depending only on $p$, 
$N$ and $T$. 
\end{lemma}
For the proof we refer to \cite{G12}. 

\mysection                                  
{Solvability of the finite difference scheme}       \label{section tools} 

In this section we study the finite difference scheme 
\eqref{scheme}-\eqref{schemeini} in the whole $\bR^d$ 
 instead of $\bG_h$, i.e., we consider

\begin{align}                         \label{wholescheme}
du_t(x)=&(L^h_tu_t(x)+f_t(x))\,dt\nonumber\\
&+(M^{h,r}_tu_t(x)+g^{r}_t(x))\,dw^{r}_t
\quad
\text{for $(t,x)\in H_T$}, 
\end{align}
with initial condition 
\begin{equation}                         \label{wholeschemeini}
u_0(x)=\psi(x), \quad\text{for $x\in \bR^d$}, 
\end{equation}
where $L^h$ and $M^{h,r}$ are defined in 
\eqref{1.26.10.12} and \eqref{2.26.10.12}. 

Let, as before, $\frak m$ be a nonnegative 
integer, $K\geq0$ be a constant, and 
make the following assumptions, 
which are somewhat weaker 
than those of \ref{assumption 1.15.9.12} 
and \ref{assumption 2.15.9.12}.

\begin{assumption}                \label{assumption 1.25.10.12}     
The functions $\fra^{\lambda\mu}$, 
$\frb^{\lambda}$,  
$\frak p^{\gamma}$ and $\frak q^{\gamma}$
and their partial derivatives in $x\in\bR^d$ 
up to order $\frak m$   
are 
continuous in $x$ 
and are bounded in magnitude by 
$K$, for all $\lambda$, $\mu\in\Lambda_1$ 
and $\gamma\in\Lambda_0$. 
\end{assumption}

\begin{assumption}                  \label{assumption 2.25.10.12}                                                 
The initial value $\psi$ 
is an $H^{\frak m}$-valued 
$\cF_0$-measurable 
random variable, and 
$f=(f_t)_{t\geq0}$ and 
$g=(g_t)_{t\geq0}$ are  
predictable processes with values in 
$H^{\frak m}$ and $H^{\frak m}(l_2)$, 
respectively, such that almost surely  
\begin{equation}                           \label{4.25.10.12}
\int_0^T|f_t|_{\frak m}^2
+|g_t|^2_{\frak m}\,dt<\infty.
\end{equation} 
\end{assumption}

\begin{definition}               \label{definition 1.25.10.12}
An $H^0$-valued continuous and adapted process 
$(\bar u_t^h)_{t\in[0,T]}$ is called a solution 
to \eqref{wholescheme}-\eqref{wholeschemeini} if 
almost surely
\begin{equation}                  \label{1.30.10.12}
(\bar u_t^h,\varphi)=(\psi,\varphi)
+\int_0^t(L^h_s\bar u_s^h+f_s,\varphi)\,ds
+\int_0^t(M^{h,r}_s\bar u_s^h+g^r_s,\varphi)\,dw^r_s
\end{equation}
for all $t\in[0,T]$ 
and $\varphi\in C_0^{\infty}(\bR^d)$. 
\end{definition}

Note that due to 
Assumption \ref{assumption 1.25.10.12} 
for each $h\neq0$ we have a constant $C$ such that 
\begin{equation}                     \label{1.26.10.13}
|L^h_t\phi|^2_{\frak l}
+\sum_{r=1}^{\infty}|M^{h,r}_t\phi|^2_{\frak l}
\leq C^2|\phi|^2_{\frak l}
\quad 
\text{for all $\phi\in H^{l}$}, 
\end{equation}
for all $t\in[0,T]$ 
and integer $\frak l\in[0,\frak m]$. 
Thus $L^h_t$ and $M^h_t=(M_t^{h,r})_{r=1}^{\infty}$ 
are bounded linear operators from $H^l$ to 
$H^l$ and to $H^l(l_2)$, 
respectively, 
such that their operator norm is bounded by 
$C$ for all $t\in[0,T]$ and $\omega\in\Omega$.  
Using this with $l=0$, 
\eqref{4.25.10.12} with $\frak m=0$ 
and that $\psi$ is 
a $\cF_0$ measurable random variable in $H^0$, 
we can see that 
$\bar u$ is a solution to 
\eqref{wholescheme}-\eqref{wholeschemeini} 
if and only if it is 
the solution of the SDE 
\begin{equation}                          \label{6.25.10.12}
\bar u_t^h=\psi+\int_0^t(L^h_s\bar u_s^h+f_s)\,ds+
\int_0^t(M^{h,r}_s\bar u_s^h+g^{r}_s)\,dw^r_s, 
\quad t\in[0,T], 
\end{equation}
in the Hilbert space $H^0$, where 
the first integral is a Bochner 
integral and the second one is a stochastic 
integral in a Hilbert 
space. Thus by a well-known theorem on SDEs 
in Hilbert spaces, 
with Lipschitz continuous coefficients, problem 
\eqref{wholescheme}-\eqref{wholeschemeini} admits a 
unique solution $\bar u^h$. 
Due to $\psi\in H^{\frak m}$, 
\eqref{4.25.10.12} and \eqref{1.26.10.13} 
with $l=\frak m$, equation \eqref{6.25.10.12} admits 
a unique continuous $H^{\frak m}$-valued solution 
by virtue of the same theorem on solvability of 
SDEs in Hilbert spaces with Lipschitz continuous 
coefficients. Consequently, $\bar u^h$ is a 
continuous $H^{\frak m}$-valued process. 
If $\frak m>d/2$ then 
by Sobolev's theorem on embedding 
$H^{\frak m}$ into $C_b(\bR^d)$, the space of continuous 
and bounded functions on $\bR^d$,  there exists 
a linear operator $J:H^{\frak m}\to C_b$  
such that $J\varphi(x)=\varphi(x)$ for almost every 
$x\in\bR^d$,  and 
$$
\sup_{\bR^d}|J\varphi|\leq N|\varphi|_{\frak m}
$$ 
for all $\varphi\in H^{\frak m}$, where $N$ 
is a constant depending only on $d$. 
One  
has also the following  lemma on the embedding 
 $H^{\frak m}\subset l_2(\mathbb G_h)$.   
 
\begin{lemma}                            \label{lemma beagyazas}                                                                     \label{lemma 2.27.4.9}
For $H^{\frak m}$, $\frak m>d/2$, we have  
$h\in(0,1)$
\begin{equation}                        \label{7.27.4.9}   
\sum_{x\in\mathbb G_h}|J\varphi(x)|^2h^d
\leq N |\varphi|_{l}^2, 
\end{equation}
where $N$ is a constant depending only on $d$. 
\end{lemma}

\begin{proof} 
This lemma is a straightforward 
consequence of Sobolev's theorem 
on embedding $W^{\frak m}_2$ functions 
on the unit ball $B_1$ of $\bR^d$ into 
$C(B_1)$, the space of continuous functions on $B_1$. 
(See, e.g., \cite{GK4}.)
\end{proof}

If $\frak m>d/2$ then by virtue of the above lemma 
the solution $(J\bar u^h_t)_{t\in[0,T]}$ of 
\eqref{wholescheme}-\eqref{wholeschemeini} 
restricted to $\bG_h$ in $x$ is a continuous 
$l_{h,2}$-valued process. 
Thus we have the following proposition. 
Remember that when $f$ and $g$ are $H^{\frak m}$ 
and $H^{\frak m}(l_2)$-valued processes 
with $\frak m>d/2$, and $\psi\in H^{\frak m}$, 
then we always take 
their continuous modifications in $x$, i.e., 
we take $\hat f=Jf$, $\hat g=Jg$ and $\hat\psi=J\psi$ 
in place of $f$, $g$ 
and $\psi$, respectively. 

\begin{proposition}           \label{proposition 1.26.10.12}
Let Assumptions \ref{assumption 1.25.10.12} 
and \ref{assumption 2.25.10.12} hold. Then 
\eqref{wholescheme}-\eqref{wholeschemeini} has a 
unique solution $\bar u^h=(\bar u^h_t)_{t\in[0,T]}$ 
for each $h\neq0$ 
in the sense of Definition \ref{definition 1.25.10.12}. 
Moreover, $\bar u^h$ is a continuous $H^{\frak m}$-valued 
process. If $\frak m>d/2$ then $\hat u^h:=J\bar u^h$, 
the continuous modification in $x$ of $\bar u^h$, 
restricted to $\bG_h$ is $u^h$, 
the continuous $l_{h,2}$-valued 
solution of  \eqref{scheme}-\eqref{schemeini}.  
\end{proposition}
\begin{proof}
Except of the last statement we have already proved 
this proposition. To prove the last statement 
assume $\frak m>d/2$. Fix a point $x$ of $\bG_h$ 
and take a nonnegative smooth function $\varphi$ with 
compact support in $\bR^d$ whose integral over $\bR^d$ 
is one. Define for each integer $n\geq1$ the function 
$\varphi_n$ by $\varphi_n(y)=n^{d}\varphi(n(y-x))$ 
for $y\in\bR^d$. 
Then by virtue of Definition \ref{definition 1.25.10.12} 
we have almost surely 
$$
(\hat u_t^h,\varphi_n)=(\hat\psi,\varphi_n)
+\int_0^t(L^h_t\hat u_t^h+\hat f_t,\varphi_n)\,dt
+\int_0^t(M^{h,r}_t\hat u_t^h+\hat g^r_t,\varphi_n)\,dw^r_t
$$
for all $t\in[0,T]$ and for all $n\geq1$. 
Letting here $n\to\infty$ we obtain 
$$
\hat u_t^h(x)=\hat\psi(x)
+\int_0^t(L^h_t\hat u_t^h(x)+\hat f_t(x))\,dt
+\int_0^t(M^{h,r}_t\hat u_t^h(x)+\hat g^r_t(x))\,dw^r_t
$$
almost surely for each $t\in[0,T]$, where almost surely 
the right-hand side is continuous in $t\in[0,T]$. 
Since $\hat u^h_t(x)$ is also continuous in $t$, we have this 
equality almost surely for all $t\in[0,T]$ 
and all $x\in\bG_h$. Moreover, 
we know that $(\hat u^h)_{t\in[0,T]}$ is a continuous and 
adapted $l_2$-valued process, 
and that the $l_2$-valued solution $u^h$ 
of \eqref{scheme}-\eqref{schemeini} is unique. 
Hence $\hat u^h=u^h$.  
\end{proof}
\medskip

Our aim now is to obtain an estimate for the 
solution $\bar u^h$ to 
\eqref{wholescheme}-\eqref{wholeschemeini} 
independently of $h$. 
To this end first 
for $u\in H^{\frak m}$, $f\in H^{\frak m}$ and 
$g=(g^{r})_{r=1}^{\infty}\in H^{\frak m+1}(l_2)$ 
we set 
$$
Q_{\frak m}^h(u,f,g,t)
=\int_{\bR^d}\sum_{\alpha\in\Lambda^{\frak m}}
2(\delta_{\alpha}^hu)\delta_{\alpha}^h(L^h_tu(x)+f(x))\,dx
$$
\begin{equation}                                \label{2.5.10.12}
+\int_{\bR^d}\sum_{\alpha\in\Lambda^{\frak m}}
\sum_{r}|\delta_{\alpha}^h(M^{h,r}_tu(x)+g^{r}(x))|^2\,dx,    
\end{equation}
and prove the following lemma. 

\begin{lemma}                                           \label{lemma 1.7.5.12}
Let Assumptions \ref{assumption 1.15.9.12}  
and 
\ref{assumption 3.7.5.12} hold. 
Then for each $h>0$ for $P\otimes dt$-almost all 
$(\omega,t)\in\Omega\times[0,T]$ we have 
\begin{equation}                                         \label{5.7.5.12}
Q_{\frak m}^h(u,f,g,t)\leq N(|u|_{\frak m}^2
+|f|_{\frak m}^2+|g|^2_{\frak m+1})
\end{equation}
for all $u,f\in H^{\frak m}$ and  
$g\in H^{\frak m+1}(l_2)$, where $N$ 
is a constant depending only on $d$, $\frak m$, 
$K$ and $|\Lambda|$. 
If in addition to the above assumptions 
$\frak p^{\lambda}=\frak q^{\lambda}=0$ 
for all $\lambda\in\Lambda_0$, then \eqref{5.7.5.12} 
holds for each 
$h\in\bR\setminus\{0\}$ 
for $P\otimes dt$-almost every $(\omega,t)$.  
\end{lemma}

\begin{proof} 
For real functions $v$ and $w$ on $\bR^d$ we write 
$v\sim w$ if their integrals over $\bR^d$ are equal. We write $v\preceq w$ 
if $v=w+F$ for a function $F$ on $\bR^d$ such that the integral of $F$ over 
$\bR^d$ can be estimated 
by the right-hand side of \eqref{5.7.5.12}. 
We use the notation $v_{\lambda}=\delta_{\lambda}v$ for functions $v$ on $\bR^d$ 
and for $\lambda\in\Lambda^k$, 
$k\geq1$.  To simplify the notation we often write $\delta_{\alpha}$, 
$\Delta_{\lambda}$ and 
$I_{\alpha}$ in place of  $\delta_{\alpha}^h$,  $\delta_{\lambda}^h$ and 
$I_{\alpha}^h$, respectively. 
Moreover, we often use the convention 
that if there is no parenthesis 
to indicate the order of operations in a formula then 
operators act 
only on the first function written after them. 
For example, $\fra \delta_{\mu}u I_{\lambda}fu$ means 
$\fra (\delta_{\mu} u)(I_{\lambda}f)u$. 

Assume first that $\frak p^{\lambda}=\frak q^{\lambda}=0$ 
for all $\lambda\in\Lambda_0$.  
Consider the case $m=0$. 
Notice that for $\lambda,\mu\in\Lambda_0$ 
we have 
$$
u\fra^{\lambda\mu}\delta_{\lambda}\delta_{\mu}u
\sim-I_{\lambda}u\fra_{\lambda}^{\lambda\mu}
u_{\mu}-u_{\lambda}I_{\lambda}\fra^{\lambda\mu}u_{\mu}. 
$$
By Lemma \ref{lemma 6.7.5.12} we get 
$-I_{\lambda}u\fra_{\lambda}^{\lambda\mu}u_{\mu}\preceq 0$,  
and using $I_{\lambda}=I+h^2\Delta_{\lambda}/2$ we have 
$$
-u_{\lambda}I_{\lambda}\fra^{\lambda\mu}u_{\mu}
=-u_{\lambda}\fra^{\lambda\mu}u_{\mu}
-\tfrac{1}{2}T_{\lambda}u(\Delta_{\lambda}\fra^{\lambda\mu})
T_{\mu}u\preceq 
-u_{\lambda}\fra^{\lambda\mu}u_{\mu}. 
$$
Using Lemma \ref{lemma 6.7.5.12} for $\lambda\in\Lambda_0$ we have 
$
ua^{\lambda0}u_{\lambda}\preceq 0. 
$
Hence we have  
\begin{equation}                                   \label{7.7.5.12}
2u\sum_{\lambda,\mu\in\Lambda_1}
\fra^{\lambda\mu}\delta_{\lambda}\delta_{\mu}u
\preceq 
-2\sum_{\lambda,\mu\in\Lambda_0}
u_{\lambda}\fra^{\lambda\mu}u_{\mu}. 
\end{equation}
Notice that by Lemma \ref{lemma 6.7.5.12} we have  
$\frb^{\lambda r}u_{\lambda}\frb^{0 r}u\preceq 0$, 
and 
$$
\frb^{\lambda r}u_{\lambda}g^{r}
\sim-\frb^{\lambda r}_{\lambda}(I_{\lambda}g^{r})u
-(I_{\lambda}b^{\lambda r})g^{\lambda r}_{\lambda}u
\preceq 0.
$$
Hence
$$
\sum_{r}|M^{h,r}u+g^{r}|^2\preceq 
\sum_{\lambda,\mu\in\Lambda_0}
\sum_{r}\frb^{\lambda r}b^{\mu r}u_{\lambda}u_{\mu}, 
$$
which together with \eqref{7.7.5.12} 
gives  
$$
2u(L^hu+f)+\sum_{r}|M^{h,r}u+g^{r}|^2
\preceq-\sum_{\lambda,\mu\in\Lambda_0}
(2\fra^{\lambda\mu}-\sum_{r}\frb^{\lambda r}
\frb^{\mu r})u_{\lambda}u_{\mu}\leq 0
$$
by virtue of Assumption \ref{assumption 3.7.5.12}. 
This proves Lemma \ref{lemma 1.7.5.12} 
for $\frak m=0$. 

Consider now the case $\frak m\geq 1$. 
Let $\alpha\in\Lambda^{\frak m}$.  
Then 
by Lemmas \ref{lemma 1.3.5.12} and \ref{lemma 6.7.5.12} 
for $\lambda\in\Lambda_0$ 
we have 
$$
u_{\alpha}\delta_{\alpha}(\fra^{\lambda0}u_{\lambda})
\preceq u_{\alpha}(I_{\alpha}\fra^{\lambda0})
\delta_{\lambda}u_{\alpha}\preceq 0.
$$
Let $\lambda,\mu\in\Lambda_0$. 
Then using Lemma \ref{lemma 1.3.5.12} we have 
\begin{align*}                                                                                                   \label{8.8.5.12}
u_{\alpha}\delta_{\alpha}
(\fra^{\lambda\mu}\delta_{\lambda}\delta_{\mu}u)
\preceq& 
u_{\alpha}\sum_{\gamma\leq \alpha,|\gamma|=1}
(I_{\alpha\setminus\gamma}
\fra^{\lambda\mu}_{\gamma})
I_{\gamma}\delta_{\alpha\setminus\gamma}
\delta_{\lambda}\delta_{\mu}u    \nonumber\\
&+ u_{\alpha}
(I_{\alpha}
\fra^{\lambda\mu})
\delta_{\lambda}\delta_{\mu}u_{\alpha}. 
\end{align*}
By virtue of \eqref{2.8.5.12} 
we get 
\begin{align*}                                   
(I_{\alpha}
\fra^{\lambda\mu})\delta_{\lambda}\delta_{\mu}u_{\alpha}
=&
\fra^{\lambda\mu}\delta_{\lambda}\delta_{\mu}u_{\alpha}
+h^2(\mathcal O_{\alpha}\fra^{\lambda\mu})
\delta_{\lambda}\delta_{\mu}u_{\alpha}                     \\
=&
\fra^{\lambda\mu}\delta_{\lambda}\delta_{\mu}u_{\alpha}
+(\mathcal O_{\alpha}\fra^{\lambda\mu})
R_{\lambda}R_{\mu}u_{\alpha}. 
\end{align*}
Thus, using also Lemma \ref{lemma 6.7.5.12}, we obtain  
\begin{align*}
 u_{\alpha}
(I_{\alpha}
\fra^{\lambda\mu})\delta_{\lambda}\delta_{\mu}u_{\alpha}
&\preceq
u_{\alpha}
\fra^{\lambda\mu}\delta_{\lambda}\delta_{\mu}u_{\alpha}
\sim-u_{\alpha\lambda}(I_{\lambda}\fra^{\lambda\mu})
u_{\alpha\mu}
-(I_{\lambda}u_{\alpha})\fra^{\lambda\mu}_{\lambda}\delta_{\mu}
u_{\alpha}\\
&\preceq -u_{\alpha\lambda}\fra^{\lambda\mu}
u_{\alpha\mu}. 
\end{align*}
For $|\gamma|=1$ and $|\beta|\leq m-1$ due to 
\eqref{2.8.5.12}   and \eqref{3.18.9.12}
we have 
$$
u_{\alpha}
(I_{\beta}
\fra^{\lambda\mu}_{\gamma})
I_{\gamma}\delta_{\beta}\delta_{\lambda}\delta_{\mu}u
= u_{\alpha}
\fra^{\lambda\mu}_{\gamma}
I_{\gamma}\delta_{\lambda}\delta_{\mu}u_{\beta}
+hu_{\alpha}\mathcal 
(\mathcal P_{\beta}\fra^{\lambda\mu}_{\gamma})I_{\gamma}\delta_{\lambda}\delta_{\mu}u_{\beta}
$$
$$
=u_{\alpha}
\fra^{\lambda\mu}_{\gamma}
I_{\gamma}\delta_{\lambda}\delta_{\mu}u_{\beta}
+u_{\alpha}(\mathcal P_{\beta}\fra_{\gamma}^{\lambda\mu})
I_{\gamma}R_{\lambda}\delta_{\mu}u_{\beta}
\preceq
u_{\alpha}
\fra^{\lambda\mu}_{\gamma}
I_{\gamma}\delta_{\lambda}\delta_{\mu}u_{\beta}
$$
$$
\sim(I_{\gamma}(u_{\alpha}
\fra^{\lambda\mu}_{\gamma}))
\delta_{\lambda}\delta_{\mu}u_{\beta} 
$$
$$
=(I_{\gamma}u_{\alpha})
\fra^{\lambda\mu}_{\gamma}
\delta_{\lambda}\delta_{\mu}u_{\beta} 
+
(I_{\gamma}u_{\alpha})
(P_{\gamma}\fra^{\lambda\mu}_{\gamma})
R_{\lambda}\delta_{\mu}u_{\beta}
+(R_{\gamma}u_{\alpha})
(\delta_{\gamma}\fra^{\lambda\mu}_{\gamma})
R_{\lambda}\delta_{\mu}u_{\beta}
$$
$$
\preceq(I_{\gamma}u_{\alpha})
\fra^{\lambda\mu}_{\gamma}
\delta_{\lambda}\delta_{\mu}u_{\beta}.
$$
Hence 
\begin{equation}                          \label{2.9.5.12}
A:=2\delta_{\alpha}^hu\delta_{\alpha}^h(L^hu+f)
\preceq A_1+A_2
\end{equation}
with 
$$
A_1=-2u_{\alpha\lambda}\fra^{\lambda\mu}
u_{\alpha\mu}, \quad 
A_2=
2\sum_{\gamma\leq\alpha,|\gamma|=1}
(I_{\gamma}u_{\alpha})\fra^{\lambda\mu}_{\gamma}
\delta_{\lambda}\delta_{\mu}u_{\alpha\setminus\gamma},  
$$
where the summation convention is used with respect 
to repeated $\lambda,\mu$ from $\Lambda_0$. 
By Lemma \ref{lemma 1.4.10.12} 
for $h>0$ we have 
\begin{equation}                                    \label{3.4.10.12}
(\delta_{\alpha}u,
\delta_{\alpha}(\frak p^{\lambda}\delta_{h,\lambda}u
-\frak q^{\lambda}\delta_{-h,\lambda}u)) 
\preceq0.
\end{equation}
Clearly
$$
|\delta_{\alpha}(
\frb^{\mu}\delta_{\mu}u+g)|^2_{l_2}
=|\delta_{\alpha}
(\frb^{\lambda}\delta_{\lambda}u)|^2_{l_2}
+|\delta_{\alpha}(g+\frb^0u)|^2_{l_2}
$$
$$
+2( \delta_{\alpha}
(\frb^{\lambda}\delta_{\lambda}u),
\delta_{\alpha}(g+\frb^0u))_{l_2}, 
$$
where repeated indices $\mu$ mean summation over 
$\mu\in\Lambda_1$ and repeated indices $\lambda$ mean summation 
over $\lambda\in\Lambda_0$. It is easy to see that 
$$
|\delta_{\alpha}(g+\frb^0u)|^2_{l_2}\preceq 0 , 
$$
$$
(\delta_{\alpha}
(\frb^{\lambda}\delta_{\lambda}u),
\delta_{\alpha}(g+\frb^0u))_{l_2}
\preceq
(\delta_{\alpha}
(\frb^{\lambda}\delta_{\lambda}u),
\delta_{\alpha}(\frb^0u))_{l_2}. 
$$
Using the Lemmas \ref{lemma 1.3.5.12} 
and \ref{lemma 6.7.5.12} 
we get 
$$
(\delta_{\alpha}
(\frb^{\lambda}\delta_{\lambda}u),
\delta_{\alpha}(\frb^0u))_{l_2}
\preceq 
((I_{\alpha}\frb^{\lambda})I_{\alpha}
(\delta_{\lambda}\frb^0)u_{\alpha},
u_{\alpha})_{l_2}\preceq0. 
$$
Hence 
\begin{equation}                             \label{8.2.10.12}
B:=|\delta_{\alpha}(
\frb^{\lambda}\delta_{\lambda}u+g)|^2_{l_2}
\preceq 
|\delta_{\alpha}
(\frb^{\lambda}\delta_{\lambda}u)|^2_{l_2}
=B_0+B_1+B_2
\end{equation}
with 
$$
B_0=|\delta_{\alpha}
(\frb^{\lambda}\delta_{\lambda}u)-
(I_{\alpha}\frb^{\lambda})\delta_{\lambda}u_{\alpha}|^2_{l_2}, 
\quad
B_1=|(I_{\alpha}\frb^{\lambda})\delta_{\lambda}u_{\alpha}|^2_{l_2}, 
$$
$$
B_2=2(\delta_{\alpha}(\frb^{\lambda}\delta_{\lambda}u)-
(I_{\alpha}\frb^{\lambda})
\delta_{\lambda}u_{\alpha},(I_{\alpha}\frb^{\lambda})
\delta_{\lambda}u_{\alpha})_{l_2}. 
$$
It is easy to notice that 
\begin{equation}                             \label{6.2.10.12}
B_0\preceq0,
\end{equation}
$$
B_1\preceq 
(I_{\alpha}\frb^{\lambda r})(I_{\alpha}\frb^{\mu r})
\delta_{\lambda}u_{\alpha}\delta_{\mu}u_{\alpha}, 
$$
\begin{equation}                               \label{11.2.10.12}
B_2\preceq2\sum_{\gamma}
(I_{\beta}\frb^{\lambda r}_{\gamma})(I_{\alpha}\frb^{\mu r})
 (I_{\gamma}\delta_{\lambda}u_{\beta})
\delta_{\mu}u_{\alpha},      
\end{equation}
where $\sum_{\gamma}$ in $B_2$ 
denotes summation over all sub-sequences $\gamma$ 
of $\alpha$ which have length $1$, 
and $\beta$ denotes $\alpha\setminus\gamma$. 
Here, and later on in the proof, 
$\lambda$ and $\mu$ are from $\Lambda_0$. 
By \eqref{2.8.5.12} we have 
$$
(I_{\alpha}\frb^{\lambda r})(I_{\alpha}\frb^{\mu r})
u_{\alpha\lambda}u_{\alpha\mu}
=  (\frb^{\lambda r}+h^2\mathcal O_{\alpha}\frb^{\lambda r})
(b^{\mu r}+h^2\mathcal O_{\alpha}\frb^{\mu r})
u_{\alpha\lambda}u_{\alpha\mu}                                     
$$
$$
= \frb^{\lambda r}\frb^{\mu r}
u_{\alpha\lambda}u_{\alpha\mu} 
+(\frb^{\lambda r}\mathcal O_{\alpha}\frb^{\mu r}
+\frb^{\mu r} 
O_{\alpha}\frb^{\lambda r})
(R_{\lambda}u_{\alpha})(R_{\mu}u_{\alpha})
$$
$$
+
(\mathcal P_{\alpha}b^{\mu r})(\mathcal P_{\alpha}\frb^{\mu r})
(R_{\lambda}u_{\alpha})(R_{\mu}u_{\alpha})  
\preceq  \frb^{\lambda r}b^{\mu r}
u_{\alpha\lambda}u_{\alpha\mu}.  
$$
Consequently, 
\begin{equation}                                             \label{3.2.10.12}
B_1
\preceq  
\sum_{\lambda,\mu\in\Lambda_0}\frb^{\lambda r}\frb^{\mu r}
u_{\alpha\lambda}u_{\alpha\mu}. 
\end{equation}
By virtue of \eqref{2.8.5.12} we have 
$$
(I_{\beta}\frb^{\lambda r}_{\gamma})
(I_{\alpha}\frb^{\mu r})
 (I_{\gamma}u_{\beta\lambda})u_{\alpha\mu}
$$
\begin{align}
=&((I+h\mathcal P_{\beta})\frb^{\lambda r}_{\gamma}))
((I+h\mathcal P_{\alpha})\frb^{\mu r})
 (I_{\gamma}u_{\beta\lambda})
u_{\alpha\mu}                                             \nonumber\\ 
=&
\frb^{\lambda r}_{\gamma}\frb^{\mu r}
(I_{\gamma}u_{\beta\lambda})u_{\alpha\mu}                  \nonumber\\
&+
(\frb^{\lambda r}_{\gamma}(\mathcal P_{\alpha}\frb^{\mu r})
+(\mathcal P_{\beta}\frb^{\lambda r}_{\gamma})\frb^{\mu r})
(I_{\gamma}u_{\beta\lambda})
R_{\mu}u_{\alpha}                                           \nonumber\\
&+
(\mathcal P_{\beta}\frb^{\lambda r}_{\gamma})
(\mathcal P_{\alpha}\frb^{\mu r})
 (I_{\gamma}R_{\lambda}u_{\beta})
R_{\mu}u_{\alpha}                                            \nonumber\\
\preceq&
\frb^{\lambda\rho}_{\gamma}\frb^{\mu r}
(I_{\gamma}u_{\beta\lambda})u_{\alpha\mu}                    \nonumber
\end{align}
for $\lambda,\mu\in\Lambda_0$. 
Thus from \eqref{11.2.10.12} we obtain 
$$
B_2\preceq2\sum_{\gamma}\sum_{\lambda,\mu\in\Lambda_0}
\frb^{\lambda\rho}_{\gamma}\frb^{\mu r}
(I_{\gamma}u_{\beta\lambda})u_{\alpha\mu}. 
$$
Hence using that by \eqref{1.8.5.12}
$$
2\frb^{\lambda r}_{\gamma}\frb^{\mu r}=
2\frb^{\lambda r}_{\gamma}I_{\gamma}\frb^{\mu r}
-2h\frb^{\lambda r}_{\gamma}P_{\gamma}\frb^{\mu r}
=\delta_{\gamma}(\frb^{\lambda r}\frb^{\mu r})
-2h\frb^{\lambda r}_{\gamma}P_{\gamma}\frb^{\mu r}, 
$$
we get 
\begin{align}
B_2\preceq& \sum_{\gamma}\sum_{\lambda,\mu\in\Lambda_0}
\delta_{\gamma}(\frb^{\lambda r}\frb^{\mu r})
(I_{\gamma}u_{\beta\lambda})u_{\alpha\mu}
                                                                   \nonumber\\
&-2\sum_{\gamma}\sum_{\lambda,\mu\in\Lambda_0}
\frb^{\lambda r}_{\gamma}(P_{\gamma}\frb^{\mu r})
(I_{\gamma}u_{\beta\lambda})R_{\mu}u_{\alpha}
                                                                    \nonumber\\
\preceq& \sum_{\gamma}\sum_{\lambda,\mu\in\Lambda_0}
(\delta_{\mu}u_{\alpha})
(\delta_{\gamma}(\frb^{\lambda r}\frb^{\mu r}))
I_{\gamma}u_{\beta\lambda}.                                           \label{2.2.10.12}
\end{align}
Taking the adjoint $\delta^{\ast}_{\mu}=-\delta_{\mu}$,  
using the Leibniz rule \eqref{8.7.5.12} and then using 
\eqref{1.8.5.12}, 
we have 
\begin{align}
(\delta_{\mu}u_{\alpha})
(\delta_{\gamma}(\frb^{\lambda r}\frb^{\mu r}))
I_{\gamma}u_{\beta\lambda}
\sim&
-u_{\alpha}(I_{\mu}\delta_{\gamma}(\frb^{\lambda r}\frb^{\mu r}))
I_{\gamma}u_{\beta\lambda\mu}                                          \nonumber\\
&
-u_{\alpha}(\delta_{\mu}\delta_{\gamma}
(\frb^{\lambda r}\frb^{\mu r}))
I_{\gamma\mu}u_{\beta\lambda}                                           \nonumber\\
\preceq& -u_{\alpha}(I_{\mu}\delta_{\gamma}
(\frb^{\lambda r}\frb^{\mu r}))
I_{\gamma}u_{\beta\lambda\mu}                                           \nonumber\\
=&-u_{\alpha}(\delta_{\gamma}
(\frb^{\lambda r}\frb^{\mu r}))
I_{\gamma}u_{\beta\lambda\mu}                                        \nonumber\\
&-u_{\alpha}(P_{\mu}\delta_{\gamma}
(\frb^{\lambda r}\frb^{\mu r}))
I_{\gamma}R_{\mu}u_{\beta\lambda}                                     \nonumber\\
\preceq& -u_{\alpha}(\delta_{\gamma}
(\frb^{\lambda r}\frb^{\mu r})) 
I_{\gamma}u_{\beta\lambda\mu}.                                        \label{1.2.10.12}
\end{align}
Taking the adjoint $I_{\gamma}^{\ast}=I_{\gamma}$ and using 
\eqref{3.18.9.12} we get 
\begin{align*}
u_{\alpha}(\delta_{\gamma}
(\frb^{\lambda r}\frb^{\mu r})) 
I_{\gamma}u_{\beta\lambda\mu}
=&
(I_{\gamma}u_{\alpha})(\delta_{\gamma}
(\frb^{\lambda r}\frb^{\mu r})) 
u_{\beta\lambda\mu}                                                \\
&+(I_{\gamma}u_{\alpha})(P_{\gamma}\delta_{\gamma}
(\frb^{\lambda r}\frb^{\mu r})) 
R_{\mu}u_{\beta\lambda}                                              \\
&+(I_{\gamma}u_{\alpha})(\delta_{\gamma}^2
(\frb^{\lambda r}\frb^{\mu r})) 
R_{\mu}u_{\beta\lambda}                                              \\
\preceq&
(I_{\gamma}u_{\alpha})(\delta_{\gamma}
(\frb^{\lambda r}\frb^{\mu r})) 
u_{\beta\lambda\mu}
\end{align*}
Using this, from \eqref{1.2.10.12} we get 
$$
(\delta_{\mu}u_{\alpha})
(\delta_{\gamma}(\frb^{\lambda r}\frb^{\mu r}))
I_{\gamma}u_{\beta\lambda}
\preceq
-(I_{\gamma}u_{\alpha})(\delta_{\gamma}
(\frb^{\lambda r}\frb^{\mu r})) 
u_{\beta\lambda\mu}.  
$$
Thus from \eqref{2.2.10.12}
we have 
\begin{equation}                             \label{4.2.10.12}
B_2\preceq-\sum_{\gamma}\sum_{\lambda,\mu\in\Lambda_0}
(I_{\gamma}u_{\alpha})
(\delta_{\gamma}(\frb^{\lambda r}\frb^{\mu r}))
u_{\beta\lambda\mu}.  
\end{equation}
Set 
$$
{\hat \fra}^{\lambda\mu}
={\fra}^{\lambda\mu}-\tfrac{1}{2}\frb^{\lambda r}\frb^{\mu r} 
\quad\text{for $\lambda,\mu\in\Lambda_0$}. 
$$
Then from \eqref{3.2.10.12} and \eqref{4.2.10.12}
we obtain 
\begin{align}
A_1+B_1&\preceq -2
\sum_{\lambda,\mu\in\Lambda_0}{\hat a}^{\lambda\mu}
u_{\alpha\lambda}u_{\alpha\mu},                        \label{9.2.10.12}\\
A_2+B_2&\preceq 2\sum_{\lambda,\mu\in\Lambda_0}\sum_{\gamma}
(I_{\gamma}u_{\alpha})
{\hat \fra}^{\lambda\mu}_{\gamma}u_{\beta\lambda\mu}.   \label{10.2.10.12}
\end{align}
By the simple inequality 
$2yz\leq \varepsilon y^2+\varepsilon^{-1}z^2$ 
and by using Lemma \ref{lemma 1.15.9.12} with 
$V=(V^{\lambda\mu})=(u_{\beta\lambda\mu})$, 
we have 
\begin{align*}
2\sum_{\lambda,\mu\in\Lambda_0}I_{\gamma}u_{\alpha}
\hat\fra^{\lambda\mu}_{\gamma}u_{\beta\lambda\mu}
\leq &
\varepsilon \big|
\sum_{\lambda,\mu\in\Lambda_0}\hat\fra^{\lambda\mu}_{\gamma}
u_{\beta\lambda\mu}\big|^2
+\varepsilon^{-1}|I_{\gamma}u_{\alpha}|^2                                    \\
\leq& \varepsilon N
 \sum_{\lambda,\mu,\eta\in\Lambda_0}\hat\fra^{\lambda\mu}
(\delta_{\lambda}u_{\beta\eta})(\delta_{\mu}u_{\beta\eta})                \\
&+\varepsilon N\sum_{\lambda,\mu\in\Lambda_0}|R_{\mu}u_{\beta\lambda}|^2
+\varepsilon^{-1}|I_{\gamma}u_{\alpha}|^2                                 \\
\preceq&
\varepsilon N
 \sum_{\lambda,\mu,\eta\in\Lambda_0}\hat\fra^{\lambda\mu}
u_{\beta\eta\lambda}u_{\beta\eta\mu}
\end{align*}
for each $\varepsilon>0$, where $N$ is a constant depending  
only on $d$, $K$, $\frak m$ and $|\Lambda|$. 
Choosing here $\varepsilon$ sufficiently small, by virtue of 
\eqref{2.9.5.12}, \eqref{8.2.10.12}, \eqref{6.2.10.12},  
\eqref{9.2.10.12} and \eqref{10.2.10.12} we have for each 
$h\neq0$ 
$$
\sum_{\alpha\in\Lambda^{\frak m}}
2(\delta_{\alpha}u)\delta_{\alpha}(L^hu+f)
+\sum_{\alpha\in\Lambda^{\frak m}}
\sum_r|\delta_{\alpha}(M^{h,r}u+g^r)|^2\preceq0, 
$$
which proves the lemma when $\frak p^{\lambda}=\frak q^{\lambda}=0$  
for all $\lambda\in\Lambda_0$. If this additional condition 
is not satisfied, then the only change in the proof is that there 
is an additional term,  
$$
A_0
=2(\delta_{\alpha}u,\delta_{\alpha}
(\frak p^{\lambda}\delta_{h,\lambda}u-\frak q^{\lambda}\delta_{-h,\lambda}u)),
$$
in the right-hand side of \eqref{2.9.5.12}, where the summation convention 
with respect to repeated $\lambda\in\Lambda_0$ is in force. 
By Lemma \ref{lemma 1.4.10.12} for each $h>0$ 
$$
A_0\preceq 0, 
$$
which completes the proof of the lemma. 
\end{proof}

Our key estimate is formulated in the following theorem, where 
$\cK_{\frak m}(T)$ is defined by \eqref{1.15.09.13}. 

\begin{theorem}                                            \label{theorem estimate}
Let Assumptions \ref{assumption 1.15.9.12} 
and 
\ref{assumption 2.15.9.12}      
hold with $\frak m\geq0$. 
Then problem \eqref{scheme}-\eqref{schemeini} 
has a 
unique $L_2$-valued solution $(\bar u^h_t)_{t\in[0,T]}$, 
which is 
a continuous $H^{\frak m}$-valued process. 
If Assumption  
\ref{assumption 3.7.5.12} also holds then for 
every $p>0$ there 
is a constant 
$N=N(T,d,K,p,\frak m,|\Lambda_0|)$ such that 
\begin{equation}                            \label{becsles}
E\sup_{t\leq T}|\bar u^h_t|^p_{\frak m}
\leq N(E|\psi|_{\frak m}^p+ E\cK^p_{\frak m}(T))
\end{equation}
for all $h>0$. If in addition to the above assumptions 
$\frak p^{\lambda}=\frak q^{\lambda}=0$ 
for all $\lambda\in\Lambda_0$, then \eqref{becsles} 
holds for all $h\in\bR\setminus\{0\}$. 
\end{theorem}

\begin{proof} 
By virtue of Proposition \ref{proposition 1.26.10.12}
we need only prove the statements on the estimate 
\eqref{becsles}. To this end 
we introduce the Hilbert norm $|v|_{\cH^{\frak m}}$ 
on $H^{\frak m}$ 
by 
$$
|v|_{\cH^{\frak m}}^2=\sum_{|\alpha|
\leq {\frak m}}|\delta_{\alpha}v|^2_0, 
$$
and apply It\^o's formula to $|u_t|_{\cH^{\frak m}}^2$ to obtain 
\begin{align*}                                         \label{1.5.10.12}
d|\bar u^h_t|^2_{\cH^{\frak m}}
=&\{2(\bar u_t, L^h_t\bar u^h_t+f_t)_{\cH^m}
+\sum_{r}|M_t^{h,r}\bar u^h_t+g^{r}_t|^2_{\cH^m}\}\,dt
+dm^h_t\\
=&Q_{\frak m}^h(\bar u^h_t,f_t,g_t)\,dt+dm^h_t, 
\end{align*}
where 
$$
dm^h_t=2(\bar u^h_t,M^{h, r}_t\bar u_t+g^{r}_t)_{\cH^{\frak m}}
\,dw^{r}_t, 
$$
and $Q^h_{\frak m}$ is defined by \eqref{2.5.10.12}. 
Since $|\varphi|_{\frak m}\leq|\varphi|_{\cH^{\frak m}}$ 
for all $\varphi\in H^{\frak m}$, 
by virtue of Lemma \ref{lemma 1.7.5.12} we get 
$$
d|\bar u^h_t|^2_{\cH^{\frak m}}
\leq N(|\bar u^h_t|^2_{\cH^{\frak m}}
+|f_t|_{\frak m}^2+|g_t|_{\frak m}^2)\,dt
+dm^h_t
$$
for each $h>0$, provided 
Assumptions \ref{assumption 1.15.9.12},   
\ref{assumption 2.15.9.12} and \ref{assumption 3.7.5.12}      
hold with $\frak m\geq0$,  
and we get this for all $h\neq0$ if 
in addition to these assumptions, 
$\frak p^{\lambda}=\frak q^{\lambda}=0$ 
also holds for every $\lambda\in\Lambda_0$. 
Clearly,
$$
|(\bar u^h,M^{h,r}\bar u+g^{r})_{\cH^{\frak m}}|
\leq
|(\bar u^h,M^{h,r})_{\cH^{\frak m}}|
+|(\bar u^h,g^{r})_{\cH^{\frak m}}|, 
$$
and by Lemma \ref{lemma 1.5.10.12}
$$
\sum_{r}|(\bar u^h,g^{r})_{\cH^{\frak m}}|^2
\leq |\bar u^h|_{\cH^{\frak m}}^2
\sum_{r}|g^{r}|_{\cH^{\frak m}}^2\leq N|\bar u^h|_{\cH^{\frak m}}^2
|g|^2_{\frak m} 
$$
with a constant $N=N(m,d,\Lambda)$. 
By Corollary \ref{corollary 1.6.10.12}
$$
\sum_{r}|(\bar u^h,M^{h,r})_{\cH^{\frak m}}|^2
\leq 
N|\bar u^h|^2_{\frak m}\leq N|\bar u^h|^2_{\cH^{\frak m}}. 
$$
with a constant $N=N(\frak m,d,K,|\Lambda_0|)$. 
Thus for $y^h_t=|\bar u^h_t|^2_{\cH^{\frak m}}$ for each 
$h>0$ we have 
that almost surely 
\begin{equation}                              \label{3.6.10.12}
y^h_t\leq N\int_0^t(y^h_s+F_s)\,ds+m^h_t
\quad\text{for all $t\in[0,T]$},
\end{equation}
 where $F_s=|f_s|^2_{\frak m}+|g_s|^2_{\frak m}$ and $m^h$ is a continuous 
local martingale for every $h\neq0$, such that 
 by virtue of Corollary \ref{corollary 1.6.10.12} we have 
for every $h\neq0$ 
\begin{align}
d\<m^h\>_t=&4\sum_{r=1}^{\infty}
|(\bar u^h_t,M^{h,r}_t\bar u^h_t+g^{r}_t)_{\cH^{\frak m}}|^2\,dt
\nonumber\\
\leq&N (|y^h_t|^2+y^h_t|g_t|_{\frak m}^2)\,dt  
\end{align}
with a constant $N=N(K,\frak m,d,|\Lambda_0|)$. 
Moreover, \eqref{3.6.10.12} holds for each $h\neq0$ if 
$\frak p^{\lambda}=\frak q^{\lambda}=0$ 
for all $\lambda\in\Lambda_0$. 
Hence we finish the proof by applying 
Lemma \ref{lemma 1.6.10.12} to 
\eqref{3.6.10.12} and using  the obvious inequality 
$|u^h|^2_{\frak m}\leq y^h$. 
\end{proof}

For $p>0$ let $\mathbb H^l_p$ denote the space 
of $H^l$-valued adapted processes $(v_t)_{t\in[0,T]}$ 
such that 
$$
|v|_{\bH^l_p}^p=E\int_0^T|v_t|^p_l\,dt<\infty.
$$
Clearly, $\mathbb H^l_p$ is a Banach space for $p\geq1$,  
and it is a Hilbert space for $p=2$. 
Our aim now is to show that if 
$E\cK^p_1(T)<\infty$ for some $p\geq2$, 
then Theorem \ref{theorem estimate} 
implies the weak convergence of $\bar u^h$ in $\bH^1_p$ to some 
$\bar u$, which is the unique solution 
of 
\begin{equation}                          \label{5.30.10.12}
du_t=(\fra_t^{\lambda\mu}
\partial_{\lambda}\partial_{\mu}u_t+f_t)\,dt
+(\frb^{\lambda r}_t\partial_{\lambda}u_t+g^r_t)\,dw_t^r, 
\quad u_0=\psi, 
\end{equation}
on $[0,T]$, where the summation 
convention is used over repeated $\lambda,\mu\in\Lambda_1$ and 
$r\in\{0,1,...\}$. 
The notion of solution is understood in 
the sense of Definition \ref{definition solution}. Thus, 
saying that a process $\bar u$ is a solution to 
\eqref{5.30.10.12} on $[0,T]$ means that $\bar u$ has a 
weakly continuous $H^1$-valued process  
from its equivalence 
class, say $u$, such that almost surely  
\begin{align}                                     
(u_t,\varphi)=&(\psi,\varphi)
+\int_0^t\{
(\partial_{\lambda}u_s,\partial_{-\mu}(\fra^{\lambda\mu}_s\varphi))
+(f_s,\varphi)\}\,ds \nonumber\\
&+\int_0^t
(\frb^{\lambda r}\partial_{\lambda}v_s+g^r_s
,\varphi)\,dw_s^r                                   \label{1.2.11.12}
\end{align}
for all $t\in[0,T]$ and $\varphi\in C_0^{\infty}(\bR^d)$.  

\begin{theorem}                                 \label{theorem convergence}                
Let Assumptions \ref{assumption 1.15.9.12}, 
\ref{assumption 2.15.9.12} 
and 
\ref{assumption 3.7.5.12} 
hold with $\frak m\geq1$. 
Let $p\geq2$ and assume that 
$E|\psi|^p_{\frak m}<\infty$ and 
$E\cK_{\frak m}^p(T)<\infty$. Then 
$\bar u^h$ converges weakly in $\bH^{l}_p$, for 
every $l\in[0,\frak m]$,  
to some $u=(u_t)_{t\in[0,T]}$ as $n\to\infty$. 
The process $u$ is the unique 
solution of \eqref{5.30.10.12} on $[0,T]$.  
Moreover, $u$ is an 
$H^{\frak m}$-valued weakly continuous process, 
it is strongly 
continuous as an $H^{\frak m-1}$-valued process, 
and 
\begin{equation}                                        \label{ubecsles}
E\sup_{t\in[0,T]}|u|^p_{l}
\leq N(E|\psi|^p_{l}+E\cK^{p}_{l}(T))
\end{equation}
for $l\in\{0,1,...,\frak m\}$, 
with a constant $N=N(K,d,p,T,\frak m,|\Lambda_0|)$. 
\end{theorem}
\begin{proof}
If $u^{(1)}$ and $u^{(2)}$ are solutions to 
of \eqref{5.30.10.12} on $[0,\tau]$ for a stopping time 
$\tau$, then using It\^o's formula for 
$y_t=|u^{(1)}_{t\wedge\tau}-u^{(2)}_{t\wedge\tau}|^2_0$, 
due to Assumptions \ref{assumption 1.15.9.12} 
through 
\ref{assumption 3.7.5.12}, by simple calculations one obtains 
that almost surely 
$$
y_t\leq N\int_0^ty_s\,ds+m_t\quad \text{for all $t\in[0,T]$}, 
$$
for a continuous local martingale $m$ starting from $0$. 
Hence the uniqueness of the solution to 
\eqref{5.30.10.12} on $[0,\tau]$ 
follows immediately for stopping times $\tau$, in particular, 
for $\tau=T$. 
Thus to show that for  $0<h_n\to0$ we have 
$\bar u^h\to u$ weakly in $\bH^{l}_p$ for 
$l=1,...,\frak m$, where $u$ is the solution to 
\eqref{5.30.10.12}, we need only prove that 
for every sequence $0<h_n\to0$ there is a subsequence 
$h_n'$ such that $\bar u^h\to u$ 
converges weakly in $\bH^{l}_p$, for 
$l=1,...,\frak m$, to the solution $u$ of \eqref{5.30.10.12} 
as $h=h_n'\to 0$.  
To this end let us consider a sequence $0<h_n\to0$.  
Then, since by virtue of Theorem \ref{theorem estimate} 
we have 
\begin{equation}                        \label{2.27.10.12}
\sup_{n}|\bar u^{h_n}|_{\frak \bH^l_p}<\infty, 
\quad l=0,1,...,\frak m, 
\end{equation}
there is a subsequence of $h_n$, 
for simplicity we denote it also 
by $h_n$, such that  
$$
\bar u^{h_n}\to u 
\quad
\text{weakly in $\bH^l$ for $l=0,1,...,\frak m$}, 
$$
for some $u\in\bH^m$. 
Fix a function  
$\varphi\in C_0^{\infty}(\bR^d)$ 
and an adapted real valued process 
$(\eta_t)_{t\in[0,T]}$ which is bounded by some 
constant $c$. 
Define the functionals $\Phi$, $\Phi_h$, $\Psi$ 
and $\Psi_h$ over $\bH^1_p$ by 
$$
\Phi(v)=E\int_0^T\eta_t\int_0^t
\sum_{\lambda,\mu\in\Lambda_1}
(\partial_{\lambda}v_s, \partial_{-\mu}(\fra^{\lambda\mu}_s\varphi))
\,ds\,dt, 
$$
$$
\Phi_h(v)=E\int_0^T\eta_t\int_0^t
\sum_{\lambda,\mu\in\Lambda_1}
(\fra_s^{\lambda\mu}\delta^h_{\lambda}\delta_{\mu}^hv_s, \varphi),
\,ds\,dt, 
$$
$$
\Psi(v)=E\int_0^T\eta_t\int_0^t
\sum_{\lambda\in\Lambda_1}
(\frb^{\lambda r}_s\partial_{\lambda}v_s,\varphi)\,dw^r_s\,dt
$$
$$
\Psi_h(v)=E\int_0^T\eta_t\int_0^t
\sum_{\lambda\in\Lambda_1}
(\frb^{\lambda r}_s\delta^h_{\lambda}v_s,\varphi)\,dw^r_s\,dt
$$
for $v\in\bH^1_p$ and $h>0$, where $\partial_{\lambda}$ 
is the identity for $\lambda=0$.
 By the Bunyakovsky-Cauchy-Schwarz and 
the Davis-Burkholder-Gundy inequalities we have
$$
\Phi(v)\leq cTE\int_0^T
\sum_{\lambda,\mu\in\Lambda_1}
|\partial_{\lambda}v_s|_0|\partial_{\mu}
(\fra^{\lambda\mu}_s\varphi)|_0
\,ds \leq N|v|_{\bH^1_p}|\varphi|_{1}
$$
\begin{align*}
\Psi(v)\leq&cTE\sup_{t\leq T}|\int_0^t
\sum_{\lambda\in\Lambda_1}
(\frb^{\lambda r}_t\partial_{\lambda}v,\varphi)\,dw^r_s|\\
\leq& 3cT
\big\{E\int_0^T\sum_{r=1}^{\infty}|\sum_{\lambda\in\Lambda_1}
|\frb^{\lambda r}_t\partial_{\lambda}v_s|_0|\varphi|_0|^2\,ds
\big\}^{1/2}\\
\leq&N|v|_{\bH^1_p}|\varphi|_1
\end{align*}
with a constant $N=N(c,K,T,d,p,|\Lambda_0|)$. Consequently, 
$\Phi$  and $\Psi$ are continuous 
linear functionals over $\bH^1_p$, 
and 
therefore 
\begin{equation}                                \label{5.29.10.12}
\lim_{n\to\infty}\Phi(\bar u^{h_n})=\Phi(\bar u),
\quad
\lim_{n\to\infty}\Psi(\bar u^{h_n})=\Psi(\bar u). 
\end{equation} 
Using the Bunyakovsky-Cauchy-Schwarz and 
the Davis-Burkholder-Gundy inequalities 
by Lemma \ref{lemma weakerror} we get 
\begin{equation}                              \label{6.29.10.12}     
|\Phi_h(\bar u^h)-\Phi(\bar u^h)|+
|\Psi_h(\bar u^h)-\Psi(\bar u^h)|
\leq Nh|u^h|_{\bH^1_p}|\varphi|_2
\end{equation}
with a constant $N=N(c,K,T,d,p,|\Lambda|_0)$. 
From \eqref{5.29.10.12} and \eqref{6.29.10.12}, 
taking into account 
\eqref{2.27.10.12} we obtain 
\begin{equation}                            \label{2.30.10.12}
\lim_{n\to\infty}\Phi_{h_n}(\bar u^{h_n})=\Phi(\bar u), 
\quad 
\lim_{n\to\infty}\Psi_{h_n}(\bar u^{h_n})=\Psi(\bar u). 
\end{equation}
Let us now multiply both sides of equation 
\eqref{1.30.10.12} by $\eta_t$ and integrate over 
$\Omega\times[0,T]$ against the measure $P\times dt$ 
to get 
\begin{align*}                  
E\int_0^T\eta_t(\bar u_t^h,\varphi)\,dt
=&E\int_0^T\eta_t(\psi,\varphi)\,dt\\
&+E\int_0^T\eta_t\int_0^t
(L^h_s\bar u_s^h+f_s,\varphi)\,ds\,dt\\
&+E\int_0^T\eta_t\int_0^t
(M^{h,r}_s\bar u_s^h+g^r_s,\varphi)\,dw^r_s\,dt. 
\end{align*}
Taking here the limit $h=h_n\to0$, 
by virtue of \eqref{2.30.10.12} 
we get
\begin{align*}                  
E\int_0^T\eta_t(\bar u_t,\varphi)\,dt
=&E\int_0^T\eta_t(\psi,\varphi)\,dt                           \\
&+E\int_0^T\eta_t\int_0^t
(\partial_{\lambda}\bar u_s,\partial_{-\mu}(a^{\lambda\mu}_s\varphi))
+(f_s,\varphi)
\,ds\,dt\\
&+E\int_0^T\eta_t\int_0^t
(b^{\lambda r}_s\partial_{\lambda}\bar u_s+g^r_s,\varphi)\,dw^r_s\,dt 
\end{align*}
for every bounded predictable process $(\eta_t)_{t\in[0,T]}$. 
Hence \eqref{1.2.11.12} holds for $P\times dt$ almost every 
$(\omega,t)\in\Omega\times[0,T]$ for each 
$\varphi\in C_0^{\infty}(\bR^d)$, which by virtue of the theorem 
on It\^o's formula from \cite{KrR} or \cite{Ro} implies that in the equivalence 
class of $u$ there is an $H^0$-valued continuous 
 process, denoted also by $u$, such that almost surely 
\eqref{1.2.11.12} holds for all $t\in[0,T]$ and 
$\varphi\in C_0^{\infty}(\bR^d)$. 
Note that 
$$
v^{\ast}_t:=\fra^{\lambda\mu}_t\partial_{\lambda}\partial_{\mu}u_t
+f_t,\quad t\in[0,T]
$$
is an $H^{\frak m-2}$-valued process and 
$$
m_t:=\psi
+\int_0^t(\frb^{\lambda r}\partial_{\lambda}u_s+g^r_s)\,dw^r_s
\quad 
t\in[0,T], 
$$
is an $H^{\frak m-1}$-valued continuous local martingale, 
such that 
$$
u_t=\int_0^tv^{\ast}_s\,ds+m_t
$$
for $P\times dt$ almost every $(\omega,t)\in\Omega\times[0,T]$, 
and almost surely 
$$
\int_0^T|u_t|_{\frak m}^2\,dt<\infty, 
\quad 
\int_0^T|u_t^{\ast}|_{\frak m-2}^2\,dt<\infty.   
$$  
Hence by considering the continuous and dense embedding 
$H^{\frak m-2}\hookrightarrow H^{\frak m-1}$ and the identification 
of $H^{\frak m-2}$ with its adjoint by the help of the inner 
product in $H^{\frak m-1}$, using the theorem on It\^o's 
formula from \cite{KrR} or \cite{Ro} again we see that the equivalence class 
of $u\in\bH^{\frak m}_p$ contains an $H^{\frak m-1}$-valued continuous 
process. It remains to prove 
that this $H^{\frak m-1}$-valued 
continuous process, which we keep denoting also by $u$, 
is almost surely an $H^{\frak m}$-valued 
weakly continuous process, 
and that 
\eqref{ubecsles} holds. 
To this end we follow an argument from 
\cite{KR}.  By the Banach-Saks theorem there is a sequence 
$v^{n}$ of finite convex combinations of elements of 
$(u^{h_n})_{n=1}^{\infty}$ such that $v^{n}\to u$ 
strongly in $\bH^{\frak m}$ as $n\to\infty$. Hence for a subsequence 
$n_k\to\infty$ we have 
$$
v^{{n}}\to u 
\quad
\text
{strongly in $H^{\frak m}$, for $P\times dt$ almost all 
$(\omega,t)\in\Omega\times[0,T]$}.  
$$
Thus there is a dense set 
$\{t_i\}_{i=1}^{\infty}$ 
in $[0,T]$ such that for all multi-indices 
$|\alpha|\leq \frak m$ 
for each set 
$\{\phi_j\}_{j=1}^{\infty}\subset C_{0}^{\infty}(\bR^d)$ 
we have almost surely 
\begin{align}                                   
(u_{t_i}, (-1)^{|\gamma|}D^{\alpha}\varphi_j)
=&(D^{\alpha}u_{t_i}, \varphi_j)         \nonumber     \\     
=&\lim_{n\to\infty}(D^{\alpha}v^{n}_{t_i}, \varphi_j)
\leq 
\lim_{n\to\infty}|u^{h_{n}}_{t_i}|_{\frak m}|\varphi_j|_0   \label{3.4.11.12} 
\end{align}
for all integers $i,j\geq1$. Since $u$ is an $H^{\frak m-1}$ 
valued continuous process, substituting $D^{\alpha}\varphi_j$ 
for $\varphi$ in \eqref{1.2.11.12} and integrating by parts,  
we can see that almost surely 
\begin{equation}                                         \label{4.4.11.12}
(D^{\alpha}u_t,\varphi_j)
\quad
\text{is continuous in $t\in[0,T]$}
\end{equation}
for all multi-indices 
$|\alpha|\leq \frak m$ and for all $j\geq1$. 
Taking $\{\varphi_j\}_{j=1}^{\infty}$ dense in the unit ball of $H^0$, 
hence we have  
that almost surely 
\begin{align}                                    
\sup_{t\in[0,T]}|D^{\alpha}u_t|_0
=&\sup_{t\in[0,T]}\sup_{j\geq1}(D^{\alpha}u_t,\varphi_j)=
\sup_{i\geq1}\sup_{j\geq1}(D^{\alpha}u_{t_i},\varphi_j)          \nonumber\\
\leq&
\sup_{i\geq1}
\lim_{n\to\infty}|u^{h_{n}}_{t_i}|_{\frak m}
\leq \liminf_{n\to\infty}\sup_{t\in[0,T]}|u^{h_{n}}_{t}|_{\frak m}.
\end{align} 
Hence by Fatou's lemma  
$$
E\sup_{t\in[0,T]}|D^{\alpha}u_t|^p_{0}
\leq \liminf_{n\to\infty}
E\sup_{t\in[0,T]}|u^{h_{n}}_{t}|_{\frak m}^p,  
$$
which together with \eqref{becsles} yield 
$$
E\sup_{t\in[0,T]}|u_t|^p_{\frak m}
\leq N (E|\psi|_{\frak m}^p+ E\cK^p_{\frak m}(T))
<\infty.
$$
Using this, from \eqref{4.4.11.12} we get that almost surely 
$u$, as an $H^{\frak m}$-valued process, is 
 weakly continuous in $t\in[0,T]$, which 
completes the proof.   
\end{proof}

\begin{proof}[Proof of Theorem \ref{theorem 1.14.10.12}]
Note first that due to 
Assumption \ref{assumption 1.12.10.12}, 
equations \eqref{equation} and \eqref{5.30.10.12} 
are equivalent. 
Set  
\begin{equation}                          \label{1.27.10.12}
\cR_{\frak m}^2(t)=|\psi|^2_{\frak m}+\mathcal K_{\frak m}^2(t)=|\psi|^2_{\frak m}+
\int_{0}^{t}|f_s|^2_{\frak m}+|g_s|^2_{\frak m+1}\,ds, 
\quad t\in[0,T], 
\end{equation}
and define 
$$
\tau_n=\inf\{t\in[0,T]:\cR_{\frak m}(t)\geq n\}, 
$$
$$
\psi^{(n)}=\psi{\bf{1}}_{\tau_n>0},
\quad f^{(n)}_t=f_t{\bf{1}}_{t\leq\tau_n}, 
\quad g^{(n)}_t=g_t{\bf{1}}_{t\leq\tau_n}.
$$
Then $\{\tau_n\}_{n=1}^{\infty}$ is a nondecreasing 
sequence of stopping times such that 
$P(\tau_n\geq T)\to1$ for $n\to\infty$, and the process 
$\cR_{\frak m}$ with $\psi^{(n)}$, 
$f^{(n)}$ and $g^{(n)}$ in place of 
$\psi$, 
$f$ and $g$, respectively, is bounded by $n$. 
Then by virtue of Theorem \ref{theorem convergence} 
for each $n$ the Cauchy problem 
\eqref{5.30.10.12} with 
${\bf 1}_A\psi^{(n)}$, 
${\bf1}_A{\bf1}_{(0,\tau]}f^{(n)}$ and 
${\bf1}_A{\bf1}_{(0,\tau]}g^{(n)}$ in place of 
$\psi$, 
$f$ and $g$, respectively, 
admits a solution $u^{(n)}$ on $[0,T]$, and because 
of the uniqueness of the solution on any stochastic 
interval, for $m\geq n$ we have $u^{(n)}_t=u^{(m)}_t$ 
almost surely for $t\in[0,\tau_n]$. 
Hence almost surely $u_t=\lim_{n\to\infty}u^{(n)}_t$ exists 
for $t\in[0,T]$ and it is the solution to \eqref{5.30.10.12}. 
Moreover, for each stopping time $\tau\leq T$ 
and set $A\in\cF_0$ problem \eqref{5.30.10.12} 
admits a unique solution $v$ on $[0,T]$ 
and $v_t={\bf1}_Au^{(n)}_t$. Thus by 
\eqref{ubecsles} for every $p\geq2$ we have 
$$
E({\bf1}_A\sup_{t\leq \tau\wedge\tau_n}
|u_t|_l^p)
\leq 
NE({\bf1}_A\cR_{l}^p(\tau_n\wedge\tau))
$$
for every sopping time $\tau\leq T$ and $A\in\cF_0$, where  
$N$ is a constant depending only on $K$, $T$, $p$, $d$, $\frak m$, 
and $|\Lambda_0|$. 
Hence by Lemma 3.2 from \cite{GK2003} for each $p>0$ 
we get 
$$
E(\sup_{t\leq \tau\wedge\tau_n}
|u_t|_l^p)
\leq 
NE(\cR_{l}^p(\tau_n\wedge\tau))
$$ 
for every stopping time $\tau\leq T$ 
with a constant $N=N(K,T,p,d,\frak m,|\Lambda_0|)$. 
Letting here $n\to\infty$ we get \eqref{1.15.10.12}  
for any $p>0$ and stopping time $\tau\leq T$. 
\end{proof}

\mysection{The coefficients of the expansion}                   \label{section coefficients}

To determine the coefficients $u^{(1)}$,...,
$u^{(k)}$ in the expansion \eqref{expansion} we assume that 
Assumptions 
\ref{assumption 1.12.10.12} 
through \ref{assumption 3.7.5.12} hold  
and consider the system of stochastic PDEs 
$$
dv^{(n)}_{t}=\big(\cL_{t}v^{(n)}_{t}+
\sum_{l=1}^{n}\tbinom{n}{l}\cL^{(l)}_{t}v^{(n-l)}_{t}\big)\,dt
$$
\begin{equation}          
                                                    \label{5.28.6}
+\big(\cM^{r}_{t}v^{(n)}_{t}+
\sum_{l=1}^{n}\tbinom{n}{l}\cM^{(l)r}_{t}v^{(n-l)}_{t}
\big)\,dw^{r}_{t}, \quad n=1,...,k,  
\end{equation} 
with initial condition $v^{(n)}_0=0$ 
for $v^{(n)}$,  $n=1,...,k$, 
where $v^{(0)}=u$  is the solution 
of \eqref{equation}-\eqref{ini}, 
\begin{equation}                                  \label{2.10.10.12}
\cL^{(0)}_{t}=\sum_{\lambda,\mu\in\Lambda}
\ga^{\lambda\mu}_{t}\partial_{\lambda}\partial_{\mu}
+\sum_{\lambda\in\Lambda_0} 
(\gp^{\lambda}_{t}-\gq^{\lambda}_{t})\partial_{\lambda} ,
\quad
\cM^{(0)r}_{t}
=\sum_{\lambda\in\Lambda}\gb^{\lambda r}_{t}\partial_{\lambda}
\end{equation}
and 
$$
\cL^{(n)}_{t}=
\sum_{\lambda,\mu\in\Lambda_{0}}\ga^{\lambda\mu}_{t}
\sum_{j=0}^{n}A_{nj}\partial_{\lambda}^{j+1}
\partial_{\mu}^{n-j+1}+(n+1)^{-1}
\sum_{\lambda\in\Lambda_{0}}
(\ga^{\lambda0}_{t}+\ga^{0\lambda}_{t})
B_{n}\partial_{\lambda}^{n+1}
$$
\begin{equation}                                   \label{3.10.10.12}
+(n+1)^{-1}\sum_{\lambda\in\Lambda_{0}}
(\gp^{\lambda}_{t}+(-1)^{n+1}\gq_t^{\lambda})
\partial_{\lambda}^{n+1},
\end{equation}
\begin{equation}                                   \label{4.10.10.12}
\cM^{(n)r}_{t}=(n+1)^{-1}\sum_{\lambda\in\Lambda_{0}}
\gb_{t}^{\lambda r}B_n\partial_{\lambda}^{n+1}
\end{equation}
for $n\geq1$. Here and later on the constants $A_{nj}$ 
and $B_n$ are defined for integers 
$n\geq0$ and $j\in[0,n]$ as follows: 
\begin{equation}                           \label{1.9.9}
B_n= \left\{ \begin{array}{ll}
0 & \mbox{if $n$ is odd}\\
1 & \mbox{if $n$ is even}
\end{array} \right. , 
\quad
A_{nj} = \left\{ \begin{array}{ll}
0 & \mbox{if $n$ or $j$ is odd}\\
\frac{n!}{(j+1)!(n-j+1)!} & \mbox{if $n$ and $j$ are even}
\end{array} \right.  .
\end{equation}

Note that, as we can see by Lemma \ref{lemma 5.27.1} below, 
for sufficiently smooth $\varphi$   
$$
L^h_t\varphi(x)\to\cL_t^{(0)}\varphi(x),
\quad 
M^h_t\varphi(x)\to\cM_t^{(0)}\varphi(x), 
$$
$$
\frac{\partial^n}{\partial h^n}L^h_t\varphi(x)
\to \cL_t^{(n)}\varphi(x), 
\quad 
\frac{\partial^n}{\partial h^n}M^h_t\varphi(x)
\to\cM_t^{(n)}\varphi(x),
$$
as $h\to0$, and the above system of SPDEs 
is obtained by differentiating 
formally $n$-times equation \eqref{equation} 
in the parameter $h$, 
and taking $h\to0$. 

The notion of solution is understood in the sense of 
Definition \ref{definition solution}. 
To formulate a suitable existence and uniqueness theorem 
we use for integers $l\geq1$ the notation $\bC^l(T)$ 
for the space of $H^{l}$-valued 
weakly continuous adapted processes 
$v=(v_t)_{t\in[0,T]}$ which are strongly continuous as 
$H^{l-1}$-valued processes.  

\begin{theorem}
                                               \label{theorem 5.28.1}
Let Assumptions 
\ref{assumption 1.12.10.12} 
through \ref{assumption 3.7.5.12}
hold with  $\frak m\geq 2k+1$ for an integer 
$k\geq1$.   
Then there is  a unique solution 
$v^{(1)},...,v^{(k)}$
to \eqref{5.28.6} with initial condition 
$v^{(1)}_{0}=...=v^{(k)}_{0}=0$. Moreover,  
$v^{(n)}\in\bC^{\frak m-2n}(T)$ 
and 
\begin{equation}
                                          \label{5.28.7}
E\sup_{t\leq T}|v^{(n)}_{t}|^{p}_{\frak m-2n}
\leq N(E|\psi|_{\frak m}^p+E\cK_{\frak m}^{p}(T))
\quad
\text{for $n=1,...,k$},  
\end{equation}
for $p>0$, with a constant $N=N(\frak m,k,T,K,p,|\Lambda_0|)$. 
If $k$ is an odd number, and 
$\gp^{\lambda}=\gq^{\lambda}=0$ for $\lambda\in\Lambda_0$, 
then Assumptions 
\ref{assumption 1.12.10.12} 
through 
\ref{assumption 3.7.5.12} need only be satisfied with 
$\frak m\geq 2k-1$ for the above conclusions to hold.  
In this case we  
have $v^{(n)}=0$ for odd $n\leq k$. 
\end{theorem}

\begin{proof} Though this theorem can be proved in the same way as 
Theorem 3.7 from \cite{G11}, for the convenience of the reader 
we present its proof here. 
 For each $n= 1,\dots,k$ the  equation 
for $v^{(n)}_{t}$
involves only the unknown functions 
$v^{(1)}$,..., $v^{(n)}$. Therefore  
we prove the theorem recursively on $n\leq k$. 
Set 
$$
S^{(n)}=\sum_{i=1}^{n}\tbinom{n}{i}\cL^{(i)}v^{(n-i)}, 
\quad 
R^{(n)\rho}=\sum_{i=1}^{n}\tbinom{n}{i}\cM^{(i)\rho} 
v^{(n-i)},
$$
and consider first the case $n=1$.   
By Theorem \ref{theorem 1.14.10.12}
we have $v^{(0)}\in\bC^{\frak m}(T)$ and 
estimate \eqref{1.15.10.12} holds. 
Since $\cM^{(1)}=0$, 
we have 
$R^{(1)}=0$, and due to 
Assumption \ref{assumption 1.15.9.12} 
we have 
$$
|\cL^{(1)}_{t}v^{(0)}_{t}|_{\frak m-2}
\leq N|v_{t}^{(0)}|_{\frak m}.  
$$
Hence 
$$
\int_0^T|S_t^{(1)}|^2_{\frak m-2}+|R_t^{(1)}|^2_{\frak m-1}\,dt
\leq N \sup_{t\leq T}|v_t^{(0)}|^2_{\frak m}. 
$$
Thus by Theorem \ref{theorem 1.14.10.12} 
there exists a unique $v^{(1)}\in\bC^{\frak m-2}(T)$ satisfying
\eqref{5.28.6} with  zero initial condition, 
and \eqref{5.28.7} holds with $n=1$. 
For $n\in\{2,...,k \} $ 
we assume that there exist $v^{(1)}$,...,$v^{(n-1)}$ 
with the asserted properties.  Then $\cM^{(1)}v^{(n-1)}=0$ and 
$$
|\cL^{(1)}_{t}v^{(n-1)}_{t}|_{\frak m-2n}
\leq N|v_{t}^{(n-1)}|_{\frak m-2n+2}
=N|v_{t}^{(n-1)}|_{\frak m-2(n-1)}, 
$$
and for $i\geq2$ 
$$
|\cL^{(i)}_{t}v^{(n-i)}_{t}|_{\frak m-2n}
\leq N|v^{(n-i)}_{t}|_{\frak m-2n+(i+2)}
= 
N|v^{(n-i)}_{t}|_{\frak m-2(n-i)}, 
$$
$$
\sum_{r=1}^{\infty}
|\cM^{(i)r}v^{(n-i)}|^2_{\frak m-2n+1}
= N|v^{(n-i)}|_{\frak m-2n+1+(i+1)}^{2}
$$
\begin{equation*}
                                              \label{5.28.9}\leq 
N|v^{(n-i)}_{t}|_{\frak m-2(n-i)}. 
\end{equation*}
It follows from the induction hypothesis that 
almost surely 
\begin{equation}
                                               \label{5.28.8}
\int_0^T|S^{(n)}_t|^2_{\frak m-2n}
+|R^{(n)}_t|_{\frak m-2n+1}^2\,dt
\leq N\sum_{i=1}^n
\sup_{t\leq T}|v^{(n-i)}_t|_{\frak m-2(n-i)}^2<\infty,    
\end{equation}
which by Theorem \ref{theorem 1.14.10.12} yields 
the existence of a unique 
$v^{(n)}\in\bC^{\frak m-2n}(T)$ 
satisfying \eqref{5.28.6} with  zero initial condition, 
and estimate \eqref{1.15.10.12} applied to $v^{(n)}$  
via \eqref{5.28.8}
gives 
\eqref{5.28.7}. The proof of the existence of 
$v^{(1)}$,...,$v^{(k)}$ with the stated properties 
is complete. Since for any given $v^{(i)}\in \bC^{\frak m-2i}(T)$, 
$i=1,...,n-1$, the solution $v^{(n)}$ 
to equation \eqref{5.28.6} with zero initial condition 
is uniquely determined, 
the uniqueness of the solution $(v^{(1)}...,v^{(k)})$ 
also holds. 

Notice that $\cM^{(n)}=0$ for odd $n\leq k$ 
by 
\eqref{1.9.9}. 
Assume now that $\gp^{\lambda}=\gq^{\lambda}=0$ for 
$\lambda\in\Lambda_{0}$.    Then also $\cL^{(n)}=0$ 
for odd $n\leq k$ by \eqref{1.9.9}.    
Hence $S^{(1)}=0$ and $R^{(1)}=0$, 
which implies $v^{(1)}=0$.    
Assume that $k\geq2$ and that for an odd $n\leq k$ 
we have $v^{(l)}=0$ for all odd $l<n\leq k$. Then 
$\cL^{(i)}v^{(n-i)}=0$ and $\cM^{(i)}v^{(n-i)}=0$ for 
all $i=1$,...,$n$, since either $i$ or $n-i$ is odd. 
Thus $S^{(n)}=0$ and $R^{(n)}=0$, and hence 
$v^{(n)}=0$ for all odd $n\leq k$. In particular,  
when $k$ is odd we need only solve \eqref{5.28.6} 
for $n=1,..,k-1$, since for $n=k$ the right-hand side of 
equation \eqref{5.28.6} is zero. 
\end{proof}

We evoke now Lemma 3.6 from \cite{G11}.
\begin{lemma}
                                                       \label{lemma 5.27.1}
Let  $n\geq0$ be an integer, let
 $\phi\in W^{n+1}_{2}$, $\psi\in W^{n+2}_{2}$, and $\lambda,
\mu\in\Lambda_0$.
Set 
$$
\delta_{0,\lambda}\phi
=\partial_{\lambda}\phi=\lambda^{i}D_{i}\phi,\quad
\partial_{\lambda\mu}=\partial_{\lambda}\partial_{\mu}.
$$
Then  
 we have
\begin{equation}
                                                      \label{5.28.1}
\frac{\partial^{n}}{(\partial h)^{n}}\delta_{h,\lambda}\phi(x)
=\int_{0}^{1}\theta^{n} \partial_{\lambda}^{n+1}\phi
(x+h\theta\lambda)\,d\theta,
\end{equation}

\begin{equation}
                                                     \label{1.13.12.9}
\frac{\partial^{n}}{(\partial h)^{n}}\delta_{\lambda}^{h}\phi(x)
=\frac{1}{2}\int_{-1}^{1}\theta^{n} \partial_{\lambda}^{n+1}\phi
(x+h\theta\lambda)\,d\theta,
\end{equation}
$$
\frac{\partial^{n}}{(\partial h)^{n}}\delta_{\lambda}^h
\delta_{\mu}^h\psi(x)
$$
\begin{equation}
                                                      \label{5.28.2}
=\frac{1}{4}\int_{-1}^{1}\int_{-1}^{1}
(\theta_{1}\partial_{\lambda}+\theta_{2}
\partial_{\mu})^{n} \partial_{\lambda \mu}
\psi(x+h(\theta_{1}\lambda+\theta_{2}\mu))\,d\theta_{1}d\theta_{2},
\end{equation}
for each $h\neq 0$, for almost all $x\in\bR^{d}$. 
Furthermore, 
\begin{equation}                                   \label{7.13.12.9}
\frac{\partial^{n}}{(\partial h)^{n}}
\delta_{h,\lambda}\phi\big |_{h=0}
=\frac{1}{n+1}\partial_{\lambda}^{n+1}\phi, 
\quad 
\frac{\partial^{n}}{(\partial h)^{n}}
\delta_{\lambda}^{h}\phi\big |_{h=0}
=\frac{B_n}{n+1}\partial_{\lambda}^{n+1}\phi, 
\end{equation}
\begin{equation}                                   \label{8.13.12.9}
\frac{\partial^{n}}{(\partial h)^{n}}\delta_{\lambda}^h
\delta_{\mu}^h\psi\big|_{h=0}
=\sum_{r=0}^{n}A_{n,r}\partial_{\lambda}^{r+1}
\partial_{\mu}^{n-r+1}\psi, 
\end{equation}
and for integers $l\geq0$, functions 
$\phi\in H^{n+2+l}$ and $\psi\in H^{n+3+l}$ 
we have 
\begin{equation}
                                             \label{5.28.3}
\big|\delta_{h,\lambda}\phi-
\sum_{i=0}^{n}
\frac{h^{i}}{(i+1)!}\partial_{\lambda}^{i+1}
\phi\big|_{l}\leq
\frac{|h|^{n+1}}{(n+2)!}
|\partial_{\lambda}^{n+2}\phi|_{l}
\end{equation}
\begin{equation}
                                                    \label{2.13.12.5.9}
\big|\delta_{\lambda}^{h}\phi-
\sum_{i=0}^{n}
\frac{h^{i}}{(i+1)!}B_i\partial_{\lambda}^{i+1}
\phi\big|_{l}\leq
\frac{|h|^{n+1}}{(n+2)!}
|\partial_{\lambda}^{n+2}\phi|_{l}
\end{equation}
\begin{equation}
                                                  \label{5.28.4}
\big|\delta_{\lambda}^{h}
\delta_{\mu}^{h}\psi-\sum_{i=0}^{n}
h^{i} \sum_{j=0}^{i}A_{i,j}
\partial_{\lambda}^{j+1}\partial_{\mu}^{i-j+1}\psi
\big|_{l}\leq N|h|^{n+1}|\psi|_{l+n+3},
\end{equation}
for every $h\neq0$, with a constant $N=N(|\lambda|,|\mu|,d,n)$, 
where $A_{i,j}$ and $B_i$ are defined by 
\eqref{1.9.9}.  
\end{lemma}

\begin{proof} 
It suffices to prove the lemma for 
$\phi,\psi\in C_0^{\infty}(\bR^d)$. 
Formulas \eqref{5.28.1} and \eqref{1.13.12.9} 
for $n=0$ can be obtained 
by applying the Newton-Leibniz formula to 
the function 
$\phi(x+\theta h\lambda)$, for $\theta\in[0,1]$ 
and $\theta\in[-1,1]$, respectively.
We get \eqref{5.28.2} for $n=0$ by applying 
\eqref{1.13.12.9} with $n=0$ twice. 
After that we obtain  
\eqref{5.28.1}--\eqref{5.28.2} 
for $n\geq1$ 
by differentiating these equations
written with $n=1$. 
Formulas in \eqref{7.13.12.9} and \eqref{8.13.12.9}
follow from those in 
 \eqref{5.28.1}--\eqref{1.13.12.9} and \eqref{5.28.2}, 
respectively. We obtain estimates \eqref{5.28.3}-\eqref{5.28.4} 
by applying Taylor's formula to the functions 
$$
F(h)=\delta_{h,\lambda}\phi(x)=\int_{0}^{1}\partial_{\lambda}
\phi(x+h\theta\lambda)\,d\theta, 
$$
and 
$$
G(h)=\delta_{\lambda}^{h}\phi(x)=\frac{1}{2}\int_{-1}^{1}\partial_{\lambda}
\phi(x+h\theta\lambda)\,d\theta   
$$
with remainder terms given in integral form, whose 
$H^l$-norm (in $x$) we estimate by Minkowski's inequality. 
For more details we refer to \cite{G11}. 
\end{proof}

Introduce the operators 
$$
\cO^{h(n)}_{t}=L^{h}_{t}-\sum_{i=0}^{n}\frac{h^{i}}{i!}
\cL^{(i)}_{t},
\quad
\cR^{h(n)r}_{t}=M^{h,r}_{t}-
\sum_{i=0}^{n}\frac{h^{i}}{i!}\cM^{(i)r}_{t}, 
\quad r\geq1
$$
for integers $n\geq0$, where $\cL^{(i)}$ and  
$\cM^{(i),r}$ are defined by \eqref{2.10.10.12} through 
\eqref{4.10.10.12}.   

\begin{corollary}
                                           \label{corollary 5.28.01}
Let $l$ be a non-negative  integer. Let 
$\fra^{\lambda\mu}$, $\frb^{\lambda}$, 
$\frak p^{\kappa}$, $\frak q^{\kappa}$ 
 and their derivatives in $x$ up to order $l$ 
be functions and be bounded 
by a constant $C$ for all 
$\lambda,\mu\in\Lambda_1$ and $\kappa\in\Lambda_0$. 
Then for every $n\geq0$ 
\begin{equation}                                \label{1.10.10.12}
|\cO^{h(n)}_{t}|_l\leq N|h|^{n+1} |\phi|_{l+n+3}
\end{equation}
\begin{equation}
                                               \label{5.28.02}
|\cR^{h(n)}_{t}\phi|^2_{l}=\sum_{\rho=1}^{\infty}|\cR^{h(n)\rho}_{t}\phi|^2_{l}
\leq N
|h|^{2n+2} |\phi|^2_{l+n+2},
\end{equation}
where $N$ stands for constants depending only on 
$n$, $d$, $l$, $C$ and $|\Lambda_0|$. 
\end{corollary}
\begin{proof}
This corollary follows immediately from 
\eqref{5.28.3} through \eqref{5.28.4}. 
\end{proof}

\mysection
{Proof of Theorem \protect\ref{theorem main}} 
                                                                             \label{section proof}

Assume that the conditions of Theorem \ref{theorem 5.28.1} 
hold and set 
\begin{equation}                          \label{1.26.10.9}
\bar r^{h}_{t}
=v^{h}_{t}-\sum_{j=1}^{k}\frac{h^{j}}{j!}v^{(j)}_{t}, 
\end{equation}
where $v^h$ is the unique $L_2$-valued solution of 
\eqref{wholescheme}-\eqref{wholeschemeini}, 
$v^{(0)}$ is the solution of \eqref{equation}-\eqref{ini}, 
and $(v^{(n)})_{n=1}^{k}$ is the solution of the system of stochastic partial differential equations 
\eqref{5.28.6}. 

\begin{lemma}
                                            \label{lemma 5.28.5}
Let Assumptions 
\ref{assumption 1.12.10.12} 
through \ref{assumption 3.7.5.12}
hold with 
$$
\frak m=l+2k+2
$$
for some integers $k\geq0$ and $l\geq0$. 
Then $r^{h}_{0}=0$, $r^{h}\in\bC^{l}(T)$, and
\begin{equation}
                                             \label{5.29.1}
d\bar r^{h}_{t}=(L^{h}_{t}\bar r^{h}_{t}+F^{h}_{t})\,dt
+(M^{h,r}_{t}\bar r^{h}_{t}
+G^{h,r}_{t})
\,dw^{r}_{t},
\end{equation}
where
$$
F^{h}_{t}:=
\sum_{j=0}^{k}\frac{h^{j}}{j!}\cO^{h(k-j)}_{t}
v^{(j)}_{t},
\quad
G^{h,\rho}_{t}:=\sum_{j=0}^{k}\frac{h^{j}}{j!}
\cR^{h(k-j)}_{t}v^{(j)}_{t}.
$$
Moreover, we have that almost surely
\begin{equation}                                \label{1.9.11.12}
\int_0^T|F^{h}_t|^2_l\,dt<\infty
\quad 
\int_0^T|G^{h}_t|^2_{l+1}\,dt<\infty.
\end{equation}
If $k\geq1$ is odd and  
$\gp^{\lambda}=\gq^{\lambda}=0$ 
for $\lambda\in\Lambda_0$ then 
the above assertions remain  
true if Assumptions 
\ref{assumption 1.12.10.12} 
through \ref{assumption 3.7.5.12}
 hold only with 
$
\frak m=l+2k
$
in place of $\frak m=l+2k+2$. 
\end{lemma}

\begin{proof} 
We have $v^{(h)}\in \bC^{l}(T)$ and $v^{(j)}\in\bC^{l}(T)$, 
for $j\leq k$ by Theorems \ref{theorem 1.14.10.12} 
and \ref{theorem 5.28.1}. 
Hence $r^{h}\in \bC^{l}(T)$. 
Using the equations for $v^{h}$ and $v^{(n)}$ for $n=0,...,k$, 
we can easily see that 
\eqref{5.29.1} holds with $\hat{F}^h$ and $\hat{G}^h$ in place of $F^h$
and $G^h$, respectively, where
$$
\hat{F}^{h}=L^{h}v^{(0)}-\cL v^{(0)}
+\sum_{1\leq j\leq k}L^{h}v^{(j)}\frac{h^j}{j!}
-\sum_{1\leq j\leq k}\cL v^{(j)}\frac{h^j}{j!}-I^{h}, 
$$
$$
G^{h,r}=M^{h,r}v^{(0)}-
\cM^{r} 
{v}^{(0)}
+\sum_{1\leq j\leq k}M^{h,r}v^{(j)}\frac{h^j}{j!}
-\sum_{1\leq j\leq k}\cM^{r}{v}^{(j)}\frac{h^j}{j!}
-J^{h,r}, 
$$
with
$$
I^{h}=\sum_{1\leq j\leq k}\sum_{i=1}^{j}
\frac{1}{i!(j-i)!}\cL^{(i)}v^{(j-i)}h^{j},
$$
$$
J^{h,r}=\sum_{1\leq j\leq k}\sum_{i=1}^{j}
\frac{1}{i!(j-i)!}\cM^{(i)r} v^{(j-i)}h^{j},  
$$
where, as usual, summations over empty sets mean zero. 
Notice that 
$$
I^{h}=\sum_{i=1}^k\sum_{j=i}^{k}
\frac{1}{i!(j-i)!}\cL^{(i)}{v}^{(j-i)}h^{j}
$$
$$
=\sum_{i=1}^{k}\sum_{l=0}^{k-i}
\frac{1}{i!l!}\cL^{(i)}v^{(l)}h^{l+i} 
=\sum_{l=0}^{k-1}\frac{h^{l}}{l!}
\sum_{i=1}^{k-l}\frac{h^i}{i!}\cL^{(i)}v^{(l)}
$$
$$
=\sum_{j=0}^{k}\frac{h^{j}}{j!}
\sum_{i=1}^{ k-j} \frac{h^i}{i!}\cL^{(i)}{v}^{(j)}, 
$$
and similarly, 
$$
J^{h,r}=\sum_{j=1}^k\sum_{i=1}^{j}
\frac{1}{i!(j-i)!}\cM^{(i)r}
{v}^{(j-i)}h^j=
\sum_{j=0}^{k}\frac{h^{j}}{j!}
\sum_{i=1}^{k-j}  \frac{h^i}{i!}\cM^{(i)r}{v}^{(j)}. 
$$
Hence we get  $\hat{F}=F$ and $\hat{G}=G$
by simple calculations, and since  
by Lemma \ref{lemma 5.27.1} for $j=0,1,...,k$, 
$(\omega,t)=\Omega\times[0,T]$ we have 
$$
|\cO^{h(k-j)}_{t}v^{(j)}_{t}|_{l}
\leq N|v^{(j)}_{t}|_{l+k-j+2}
\leq N|v^{(j)}_{t}|_{\frak m-2j}, 
$$
$$
|\cR^{h(k-j)}_{t}v^{(j)}_{t}|_{l}
\leq N|v^{(j)}_{t}|_{l+k-j+1}
\leq N|v^{(j)}_{t}|_{\frak m-2j-1},  
$$
\eqref{1.9.11.12} follows y
by virtue of Theorems \ref{theorem 1.14.10.12}  
and \ref{theorem 5.28.1}. The last statement follows from 
the fact that 
$v^{(k)}=0$ for odd $k$ when $\gp^{\lambda}=\gq^{\lambda}=0$ 
for $\lambda\in\Lambda_0$. 
\end{proof}

Now we present a theorem which
is more general than Theorem 
\ref{theorem main}.  

\begin{theorem}
                                                           \label{theorem 5.29.1}
Let Assumptions \ref{assumption 1.12.10.12} 
through \ref{assumption 3.7.5.12} 
 hold with 
\begin{equation}                                           \label{1.16.12.9}
\frak m=l+2k+3
\end{equation}
for some integer $k\geq0$. 
Then for for $h>0$ and any $p>0$ we have 
\begin{equation}
                                                      \label{5.29.2}
E\sup_{t\leq T}|\bar r^{h}_{t}|^{p}_{l}
\leq N|h|^{p(k+1)}(E|\psi|^p_{\frak m}+E\cK_{\frak m}^{p}(T)) 
\end{equation}
 with a constant  $N=N(T,K,l,d,k,p,|\Lambda_0|)$. 
Moreover, if $\gp^{\lambda}=\gq^{\lambda}=0$ for 
$\lambda\in\Lambda_0$, then 
$v^{(j)}=0$ in 
\eqref{1.26.10.9} for odd $j\leq k$, and if 
$k$ is odd then it is sufficient to assume 
$\frak m= l+2k+2$ instead of \eqref{1.16.12.9} 
to have estimate \eqref{5.29.2}. 
\end{theorem}   

\begin{proof} 
Let $p>0$ such that 
$E|\psi|^p_{\frak m}+E\cK_{\frak m}^{p}(T)<\infty$. 
Lemma  \ref{lemma 5.28.5}
and Theorem \ref{theorem estimate}
yield  
\begin{equation}
                                                                   \label{5.29.4}
 E\sup_{t\leq T}|\bar r^{h}_{t}|_{l}^p
 \leq N(|F^{h}|_{\bH^l_p}^p
+|G^{h}|^{p}_{\bH^{l+1}_p})
\end{equation}
with a constant $N=N(T,k,l,d,K,|\Lambda_0|)$. 
Let \eqref{1.16.12.9} hold. Then for $j=0,...,k$
\begin{equation*}                                                \label{3.16.9.12}
l+k-j+3\leq \frak m-2j, 
\end{equation*}
and 
by Corollary \ref{corollary 5.28.01} we have
\begin{align}                                                                      
|\cO^{h(k-j)}_{t}v^{(j)}_{t}|_{l}+|\cR^{h(k-j)}_{t}v^{(j)}_{t}|_{l+1}
&\leq N|h|^{k-j+1}|v^{(j)}_{t}|_{l +k-j+3}               \nonumber \\     
&\leq N
|h|^{k-j+1} |v^{(j)}|_{\frak m-2j}.                    \label{4.16.9.12}
\end{align}
Hence, using Theorem \ref{theorem 5.28.1}
we see that
$$
|F^{h}|_{\bH^l_p}^p
+|G^{h}|^{p}_{\bH^{l+1}_p}
\leq N|h|^{2(k+1)}(E|\psi|^p_{\frak m}+E\cK_{\frak m}^{p}) , 
$$
which by \eqref{5.29.4} 
implies estimate \eqref{5.29.2}.  
If $p^{\lambda}=q^{\lambda}=0$ for $\lambda\in\Lambda_0$ 
then by Theorem \ref{theorem 5.28.1} we know that 
$v^{(j)}=0$ for odd $j\leq k$. In particular,  
$v^{(k)}=0$ for odd $k$ and 
\eqref{4.16.9.12}  obviously holds for $j=k$ and  
to have it also for $j\leq k-1$  we need only 
$\frak m=l+2k+2$. 
\end{proof}

Set $r^{h}_t=J\bar r^{h}_t$, where $J:H^l\to C_b$, 
is the Sobolev embedding operator from $H^l$ to $C_b$ 
for $l>d/2$. 
Set $\Lambda^0=\{0\}$ and recall that $\delta_{h,0}$ 
is the identity operator and 
$
\delta_{h,\lambda}=\delta_{h,\lambda_1}
\cdot...\cdot \delta_{h,\lambda_n} 
$
for $(\lambda_1,\dots,\lambda_n)\in \Lambda^n$, $n\geq1$.   
Then we have the following corollary of 
Theorem \ref{theorem 5.29.1}

\begin{corollary}                                                            \label{corollary 1.12.10.9}
Let $l>n+d/2$ for an integer $n\geq0$ 
in Theorem \ref{theorem 5.29.1}. Then 
for $\lambda\in\Lambda^n$ we 
have 
\begin{equation}                                      \label{1.23.11.12}
E\sup_{t\in[0,T]}\sup_{x\in\bR^d}
|\delta_{h,\lambda} r^h_{t}( x)|^p
\leq 
Nh^{p(k+1)}E(|\psi|_{\frak m}^p+\cK_{\frak m}^{p}), 
\end{equation}
\begin{equation}                                       \label{2.23.11.12}
E\sup_{t\in[0,T]}
\big\{\sum_{x\in\mathbb G_h}
|\delta_{h,\lambda}r^h_{t}( x)|^2h^d
\big\}^{p/2}
\leq Nh^{p(k+1)}E(|\psi|_{\frak m}^p+\cK_{\frak m}^{p})
\end{equation} 
for $h>0$ with a constant $N=N(d,n,k,K,T,p,|\Lambda_0|)$.  
\end{corollary}

\begin{proof} 
Set $j=l-n$. Then $j>d/2$ and by Sobolev's theorem on embedding 
$H^{j}$ into $C_b$ and by 
Lemma  \ref{lemma 1.5.10.12}, 
from Theorem \ref{theorem 5.29.1} we get 
$$
E\sup_{t\in[0,T]}\sup_{x\in\bR^d}|\delta_{h,\lambda}r^h_{t}( x)|^p
\leq C_1E\sup_{t\in[0,T]}|\delta_{h,\lambda}r^h_{t}|_{j}^p
$$
$$
\leq C_2E\sup_{t\in[0,T]}|r^h_{t}|_{l}^p
\leq 
Nh^{p(k+1)}E(|\psi|^p_{\frak m}+\cK_{m}^{p}). 
$$
Similarly, by Lemma \ref{lemma 2.27.4.9} and 
Lemma  \ref{lemma 1.5.10.12}
$$
E\sup_{t\in[0,T]}
\big\{\sum_{x\in\mathbb G_h}|\delta_{h,\lambda}r^h_{t}( x)|^2h^d\big\}^{p/2}
\leq C_1E\sup_{t\in[0,T]}|\delta_{h,\lambda}r^h_{t}|_{l}^p
$$
$$
\leq C_2E\sup_{t\in[0,T]}|R^h_{t}|_{l}^p
\leq 
Nh^{p(k+1)}E(|\psi|^p_{\frak m}+\cK_{\frak m}^{p}). 
$$
\end{proof}

Now we can easily see that Theorem \ref{theorem main} 
follows from the above corollary.

Let the conditions of Theorem \ref{theorem 5.29.1} hold 
with an integer $l>n+d/2$ for some integer $n\geq0$ 
and 
define 
$$
r^h=J\bar r^h, 
\quad 
u^h=Jv^h,\quad u^{(j)}=Jv^{(j)}, \quad j=0,...,k, 
$$
$v^h$ is a continuous 
$H^l$-valued process, and 
by 
Theorem \ref{theorem 5.28.1}
$v^{(j)}$, $j=1,2,...,k$,  are 
$H^{m-2k}$-valued continuous processes. 
By Proposition \ref{proposition 1.26.10.12} 
we know that $u^h$ restricted to $\bG_h$ 
is the unique $l_{2,h}$-valued solution of 
the finite difference scheme and from the previous corollary 
we have that for each $h>0$ almost surely 
\begin{equation}                    \label{3.23.11.12}
\delta_{h,\lambda}u^h_t(x)
=\sum_{j=0}^k\frac{h^j}{j!}\delta_{h,\lambda}u^{(j)}_t(x)
+h^{k+1}\delta_{h,\lambda}r_t^h(x),
\end{equation} 
for all $t\in[0,T]$ and $x\in\bG_h$,  and 
\eqref{1.23.11.12}-\eqref{2.23.11.12} hold 
for any $\lambda\in\Lambda^n$, for integers $n\geq0$ 
such that $l>n+d/2$. Moreover, by 
Theorem \ref{theorem 5.29.1} we have that $u^{(j)}=0$ 
for odd $j\leq k$ provided $\frak p^{\lambda}=\frak q^{\lambda}=0$ 
for $\lambda\in\Lambda_0$, and in this case for odd $k$ it is sufficient 
to assume that Assumptions \ref{assumption 1.12.10.12} 
through \ref{assumption 3.7.5.12} hold with $\frak m\geq l+k+2$ 
and $l>n+d/2$, to have \eqref{1.16.12.9} 
and estimates for all $h\neq0$. 
Hence Theorem \ref{theorem main}  follows immediately.

\end{document}